\documentclass[reqno]{amsart}
\usepackage{amssymb,amsmath,amsthm}
\usepackage{enumitem}

\makeatletter
\@namedef{subjclassname@2020}{\textup{2020} Mathematics Subject Classification}
\makeatother

\theoremstyle{plain}
\newtheorem{theorem}{Theorem}[section]
\newtheorem{lemma}[theorem]{Lemma}
\newtheorem{corollary}[theorem]{Corollary}
\newtheorem{proposition}[theorem]{Proposition}

\theoremstyle{definition}

\newtheorem{remark}[theorem]{Remark}

\newcommand{\bzero}{\boldsymbol{0}}
\newcommand{\balpha}{\boldsymbol{\alpha}}
\newcommand{\bbeta}{\boldsymbol{\beta}}
\newcommand{\bdelta}{\boldsymbol{\delta}}
\newcommand{\blambda}{\boldsymbol{\lambda}}

\newcommand{\bc}{\mathbf{c}}
\newcommand{\bd}{\mathbf{d}}

\newcommand{\bg}{\mathbf{g}}

\newcommand{\bq}{\mathbf{q}}

\newcommand{\bt}{\mathbf{t}}
\newcommand{\bu}{\mathbf{u}}
\newcommand{\bv}{\mathbf{v}}
\newcommand{\bw}{\mathbf{w}}

\newcommand{\by}{\mathbf{y}}
\newcommand{\bz}{\mathbf{z}}

\newcommand{\D}{\mathbf D}

\newcommand{\row}{\operatorname{row}}
\newcommand{\occ}{\operatorname{occ}}

\begin{document}
\title[The finite basis problem for $S_c(W)$, $M_c(W)$ and $M(W)$]
{The finite basis problem for flat semirings $S_c(W)$, $M_c(W)$ and $M(W)$}

\author{Zidong Gao}
\address{School of Mathematics, Northwest University, Xi'an, 710127, Shaanxi, P.R. China}
\email{zidonggao@yeah.net}

\author{Miaomiao Ren}
\address{School of Mathematics, Northwest University, Xi'an, 710127, Shaanxi, P.R. China}
\email{miaomiaoren@yeah.net}

\author{Yilin Zhou}
\address{School of Mathematics, Northwest University, Xi'an, 710127, Shaanxi, P.R. China}
\email{zhou990619@163.com}


\begin{abstract}
We study the finite basis problem for flat semirings of the forms
\(S_c(W)\), \(M_c(W)\), and \(M(W)\), where \(W\) is a nonempty set of
words in a free commutative semigroup, a free commutative monoid, and a
free monoid, respectively. We completely classify such flat semirings with
respect to the finite basis property, allowing \(W\) to be infinite. We
prove that \(S_c(W)\) is finitely based if and only if every word in \(W\)
is either a cube of a letter or has length at most two, whereas
\(M_c(W)\) and \(M(W)\) are finitely based if and only if \(W\) consists
solely of the empty word.
As applications, we recover the nonfinite
basability of \(\flat(\mathbb{Z})\) and the max-plus semiring
\((\mathbb{Z},\max,+)\).
\end{abstract}

\maketitle
\section{Introduction and preliminaries}
An \emph{additively idempotent semiring} (or ai-semiring for short) is
an algebra $(S, +, \cdot)$ such that
\begin{itemize}
\item the additive reduct $(S, +)$ is a commutative idempotent semigroup;

\item the multiplicative reduct $(S, \cdot)$ is a semigroup;

\item the distributive laws hold:
\[
x(y+z)\approx xy+xz, \quad (x+y)z\approx xz+yz.
\]
\end{itemize}

The additive reduct of an ai-semiring is a semilattice,
and the order $\leq$ on an ai-semiring defined by
\[
a \leq b\;\Leftrightarrow\; a+b=b
\]
is compatible with both addition and multiplication.
Consequently, an ai-semiring is often called a \emph{semilattice-ordered semigroup} (see \cite{kp}).
Whenever an order on an ai-semiring is mentioned, it refers to this order.
For an ai-semiring $S$, if $J$ is both a multiplicative ideal and an additive order filter of $S$,
we may form the ideal quotient $S/J$ by collapsing all elements of $J$ to a single element,
while leaving all elements outside $J$ as distinct singleton classes.
In this quotient algebra, $J$ serves as the multiplicative zero and the additive maximum element.

The class of ai-semirings includes many well-known and important examples,
such as the max-plus algebra~\cite{aei}, the Boolean algebra~\cite{bs},
and the semiring of all binary relations on a set~\cite{dol09}.
These and other similar algebras form the algebraic foundation of tropical geometry and optimization~\cite{ms},
underlie classical logic and automata theory~\cite{enderton2001, hopcroft2007},
and play a key role in relational algebra and program verification~\cite{kozen2008, maddux2006}.

A \emph{variety} of ai-semirings is a class of ai-semirings
closed under taking subalgebras, homomorphic images, and arbitrary direct products.
Birkhoff's celebrated theorem tells us that a class of ai-semirings is a variety if and only if it is an equational class;
that is, the class of all ai-semirings satisfying a certain set of identities.
A variety of ai-semirings is \emph{finitely based} if it can be defined by a finite set of identities;
otherwise, it is \emph{nonfinitely based}.
An ai-semiring $S$ is finitely based (resp., nonfinitely based)
if the variety $\mathsf{V}(S)$ it generates is finitely based (resp., nonfinitely based).

The \emph{finite basis problem} for a class of ai-semirings,
one of the most important problems in universal algebra,
concerns the classification of
its members according to whether they are finitely based.
Over the past two decades, this problem has attracted considerable attention and seen substantial progress
(see~\cite{dol07, gpz05, pas05, jac:flat, sr, jrz, rjzl, yr2602, yrg, rlyc, rlzc, ryy, yrzs}).

Pastijn et al.~\cite{gpz05, pas05} showed that every ai-semiring satisfying the identity $x^2\approx x$ is finitely based.
Dolinka~\cite{dol07} found the first example of a nonfinitely based finite ai-semiring.
Jackson~\cite{jac:flat} solved the finite basis problem for the flat extensions of finite groups,
thereby providing infinitely many nonfinitely based finite ai-semirings.
Shao and Ren~\cite{sr} proved that every ai-semiring in the variety generated by all ai-semirings of order two is finitely based.
Jackson et al.~\cite{jrz} and Zhao et al.~\cite{zrc} classified three-element ai-semirings with respect to the finite basis property.

Moreover, Volkov~\cite{vol21} and Jackson et al.~\cite{jrz} independently
resolved the finite basis problem for the ai-semiring whose multiplicative reduct is the six-element Brandt monoid.
Ren et al.~\cite{rjzl} provided an infinite number of minimal nonfinitely based ai-semiring varieties.
Very recently, the finite basis problem for four-element ai-semirings has been nearly completed, with only two algebras remaining (see \cite{yr2602, yrg, rlyc, rlzc, ryy, yrzs}).

In much of the above work, flat semirings have played an important and even decisive role.
By a \emph{flat semiring} we mean a nontrivial ai-semiring \(S\) with a
multiplicative zero \(0\) that satisfies \(a+b=0\) for all distinct
\(a,b\in S\). Equivalently, an ai-semiring is flat if and only if its
additive order has height one and its top element is the multiplicative
zero \(0\).
As observed in~\cite[Lemma 2.2]{jrz} (see also~\cite[Lemma 4.1.1]{ek}),
a semigroup with zero \(0\) becomes a flat semiring if and only if it is \(0\)-cancellative; that is, for all $a, b, c \in S$,
\[
ab = ac \neq 0 \;\Rightarrow\; b = c \quad\text{and}\quad ab = cb \neq 0 \;\Rightarrow\; a = c.
\]

Let $S$ be a cancellative semigroup. Then $S^0$ is $0$-cancellative and hence becomes a flat semiring,
which is called the \emph{flat extension} of $S$ and denoted by $\flat(S)$.
Since the element $0$ is both the additive maximum and the multiplicative zero,
we sometimes use $\infty$ in place of $0$, especially when the symbol $0$ might cause confusion.
For example, when considering the flat extension of the commutative group $\mathbb Z$,
the underlying set of $\flat(\mathbb Z)$ is by default taken to be $\mathbb Z\cup\{\infty\}$.

Let $k \geq 1$ be an integer.
A flat semiring is \emph{$k$-nil} if the $k$th power of each element is zero.
A flat semiring is \emph{$k$-nilpotent} if the product of any $k$ elements is zero.
A flat semiring is \emph{nilpotent} if it is $k$-nilpotent for some $k \geq 1$.
It is easy to see that every $k$-nilpotent flat semiring is $k$-nil;
the converse, however, does not hold in general.

It is easy to see that there is a one-to-one, inclusion-preserving
correspondence between semiring congruences and multiplicative ideals on a flat semiring.
This correspondence hinges on the fact that, in a flat semiring,
every multiplicative ideal naturally determines a semiring congruence, since it is an additive filter.
Conversely, every semiring congruence on a flat semiring with zero $0$
has its $0$-congruence class as a multiplicative ideal.
This yields the desired one-to-one order-preserving correspondence.

From this observation, we immediately obtain the following subdirect representation.
Let $S$ be a flat semiring with zero $0$, and let $(S_i)_{i\in I}$ be a family of flat semirings.
Then $S$ is isomorphic to a subdirect product of $(S_i)_{i\in I}$
if and only if there exists a family $(J_i)_{i\in I}$ of multiplicative ideals of $S$
such that $S/J_i$ is isomorphic to $S_i$ for each $i \in I$, and the intersection $\bigcap_{i\in I} J_i = \{0\}$.

Next, we introduce an important class of flat semirings. Let \(W\) be a
nonempty set of words in the free semigroup \(X^+\) over \(X\), and let
\(W^{\leq}\) denote the set of all nonempty subwords of words in \(W\).
Let \(S(W)=W^{\leq}\cup\{0\}\), where \(0\) is a new element.
Define a multiplication $\cdot$ on \(S(W)\) by, for \(\mathbf u,\mathbf v\in W^{\leq}\),
\[
\mathbf u\cdot\mathbf v=
\begin{cases}
\mathbf u\mathbf v, & \text{if }\mathbf u\mathbf v\in W^{\leq},\\
0, & \text{otherwise},
\end{cases}
\]
\[
0\cdot 0=0\cdot\mathbf u=\mathbf u\cdot 0=0.
\]
The resulting algebra is a \(0\)-cancellative semigroup, and hence becomes a flat semiring.
If we start the same construction from the free monoid $X^* = X^+\cup \{1\}$ instead,
with $W^{\leq}$ denoting the set of all subwords of words in $W$,
then the corresponding flat semiring is denoted by $M(W)$.
Repeating the construction with the free commutative semigroup $X_c^+$ (or the free commutative monoid $X_c^*$) yields the commutative analogues $S_c(W)$ and $M_c(W)$.

From the construction it follows immediately that
\(S_c(W)=S_c(W^{\leq})\). If \(W\) consists of a single word \(\mathbf w\),
we shall use \(S_c(\mathbf w)\) to denote \(S_c(W)\).
Let \(n\geq 1\) be an integer. It is easy to see that \(S_c(W)\) is
\(n\)-nilpotent if and only if every word in \(W\) has length less than
\(n\), and that \(S_c(a_1\cdots a_n)\) is \((n+1)\)-nilpotent but is not
\(n\)-nilpotent.

If $W$ consists solely of the empty word, then $M_c(W)$ coincides with $M(W)$,
is a two-element ai-semiring $M_2$, and hence is finitely based (see \cite{sr}).
Jackson et al.~\cite{jrz} proved that, when $W$ is finite and contains at least one nonempty word,
$M_c(W)$ and $M(W)$ are always nonfinitely based, while $S_c(W)$ is
nonfinitely based in most cases. Subsequently, Wu et al.~\cite{wzr}
completely solved the finite basis problem for finite flat semirings of
the form $S_c(W)$. The aim of the present paper is to solve this
problem completely for arbitrary (not necessarily finite) $S_c(W)$,
$M_c(W)$ and $M(W)$.

We conclude this section by recalling some terminology concerning identities.
From \cite[Theorem 2.5]{kp} we know that
the set $P_f(X^+)$ of all nonempty finite subsets of $X^+$,
equipped with addition as set-theoretic union and multiplication as elementwise product,
is the free ai-semiring over $X$. Consequently, an \emph{ai-semiring term} (or simply a \emph{term}) is a formal sum of finitely many words in $X^+$.
(In what follows, terms are denoted by bold lowercase letters \(\mathbf{u}, \mathbf{v}, \mathbf{w}, \dots\),
while ordinary lowercase letters \(x, y, z, \dots\) stand for letters.)
The order of the summands in the formal sum is irrelevant,
and repeated occurrences of the same word are identified with a single occurrence.

An \emph{ai-semiring identity} (or simply an \emph{identity}) is
a formal expression of the form \(\bu \approx \bv\), where \(\bu\)
and \(\bv\) are polynomials. We say that an ai-semiring \(S\)
\emph{satisfies} an identity \(\bu \approx \bv\) over
\(X=\{x_1,\ldots,x_n\}\), denoted \(S\vDash \bu\approx \bv\), if
\[
\bu(a_1, a_2, \dots, a_n) = \bv(a_1, a_2, \dots, a_n)
\]
for all \(a_1, a_2, \dots, a_n \in S\). Equivalently,
\(\varphi(\bu) = \varphi(\bv)\) for every semiring homomorphism
\(\varphi \colon P_f(X^+) \to S\). Such a homomorphism is also called
an \emph{assignment} (or a \emph{valuation}), and is uniquely
determined by the images of the elements of \(X\); for convenience,
we may denote it simply by \(\varphi \colon X \to S\).
We also use  \(S \nvDash \bu\approx \bv\) to indicate that \(S\) does not satisfy $\bu\approx \bv$.

Suppose that $\bu = \bu_1 + \bu_2 + \cdots + \bu_m$ and $\bv = \bv_1 + \bv_2 + \cdots + \bv_n$ are two ai-semiring terms.
Then it is easy to verify that $S$ satisfies the identity $\bu \approx \bv$ if and only if
it satisfies the identities $\bu \approx \bu + \bv_i$ and $\bv \approx \bv + \bu_j$ for all $1 \leq i \leq n$, $1 \leq j \leq m$.
Consequently, $S$ lies in an ai-semiring variety $\mathcal{V}$  if and only if
it satisfies every nontrivial identity of $\mathcal{V}$ of the form $\bu \approx \bu + \bq$,
where $\bu$ is a polynomial and $\bq$ is a word.

For a word \(\bw\) in \(X^+\), let \(c(\bw)\) denote the
\emph{content} of \(\bw\), that is, the set of letters occurring in
\(\bw\), and let \(\ell(\bw)\) be the \emph{length} of \(\bw\), that is,
the number of variables occurring in \(\bw\) counting multiplicities.
We also write \(\occ(x,\bw)\) for the number of occurrences of a
variable \(x\) in a word \(\bw\). For an ai-semiring term $\bu = \bu_1+\cdots+\bu_m$, where $\bu_i\in X^+$,
we let $c(\bu) = \bigcup_{1\leq i\leq m} c(\bu_i)$.

The paper is organized as follows.
In Section~\ref{sec:scw}, we provide some properties of the flat semirings $S_c(W)$.
In Section~\ref{sec:NFB}, we use hypergraph semirings to establish a sufficient condition
for nonfinite basability and apply it to certain flat semirings of the forms $S_c(W)$, $M_c(W)$, and $M(W)$.
In Section~\ref{sec:flatn}, we establish another sufficient condition via a different approach,
and apply it to resolve the finite basis problem for the remaining algebras.
In Section~\ref{sec:conclu}, we summarize the paper and discuss some open problems.

\section{The flat semirings $S_c(W)$}\label{sec:scw}
In this section, we first establish the key properties of the flat semirings $S_c(W)$,
and then apply them to solve the finite basis problem for nilpotent flat semirings of the form $S_c(W)$.
For a class $\mathcal{K}$ of ai-semirings,
we denote by $\mathsf{V}(\mathcal{K})$ the variety generated by $\mathcal{K}$;
that is, the smallest variety containing $\mathcal{K}$.

\begin{proposition}\label{WiW}
Let $\{W_i\}_{i \in I}$ be a family of nonempty subsets of $X_c^+$. Then
\[
\mathsf{V}\!\left(S_c\!\left(\bigcup_{i \in I} W_i\right)\right) = \mathsf{V}\bigl( S_c(W_i) \mid i \in I \bigr).
\]
\end{proposition}
\begin{proof}
Let $W$ denote the union $\bigcup_{i \in I} W_i$. For each $i \in I$, define $J_i = S_c(W) \setminus W_i^{\leq}$.
It is straightforward to verify that $J_i$ is a multiplicative ideal of $S_c(W)$,
and that the Rees quotient $S_c(W)/J_i$ is isomorphic to $S_c(W_i)$.
Since $\bigcap_{i \in I} J_i = \{0\}$, it follows that $S_c(W)$ is isomorphic to a subdirect product of $(S_c(W_i))_{i \in I}$.
Therefore, $\mathsf{V}(S_c(W)) = \mathsf{V}\bigl( S_c(W_i) \mid i \in I \bigr)$, as required.
\end{proof}

\begin{corollary}\label{wwi}
Let $W$ be a nonempty set of words in $X^+_c$. Then
\[
{\mathsf V}(S_c(W))={\mathsf V}(S_c(\bw) \mid \bw\in W)
\]
and
\[
{\mathsf V}(S_c(W))={\mathsf V}(S_c(\bw) \mid \bw\in W^{\leq}).
\]
In particular, $S_c(\bw)$ belongs to $\mathsf{V}(S_c(W))$ for every $\bw\in W^{\leq}$.
\end{corollary}
\begin{proof}
The first equality follows directly from Proposition~\ref{WiW}.
The second one follows from Proposition~\ref{WiW} together with the fact that $S_c(W)=S_c(W^{\leq})$.
\end{proof}

Corollary~\ref{wwi} implies that $S_c(\bw_1)\in \mathsf{V}(S_c(\bw_2))$ if $\bw_1$ is a subword of $\bw_2$.
However, the converse does not hold; counterexamples can be easily constructed using the following two lemmas,
which are due to Jackson et al.~\cite[Proposition 2.7]{jrz} and Wu et al.~\cite[Proposition 2.7]{wzr}, respectively.

\begin{lemma}[{\cite[Proposition 2.7]{jrz}}]\label{lemsc1}
Let $n\geq 1$ be an integer, and let $\bw$ be a word in $X_c^+$
such that $\ell(\bw)=n$ and $|c(\bw)| \geq 2$.
Then $S_c(a_1\cdots a_n)$ belongs to $\mathsf{V}(S_c(\bw))$.
\end{lemma}

\begin{lemma}[{\cite[Proposition 2.7]{wzr}}]\label{lemsc2}
Let $n\geq 1$ be an integer.
Then $S_c(a_1\cdots a_n)$ belongs to $\mathsf{V}(S(a^{n+1}))$.
\end{lemma}

From the above two lemmas, we can obtain a characterization of when $S_c(W)$ is nilpotent.

\begin{proposition}\label{nilpotentscw}
Let \(W\) be a nonempty set of words in \(X^+_c\). Then
\(S_c(W)\) is nilpotent if and only if \(\mathsf{V}(S_c(W))\) does not contain \(S_c(a_1\cdots a_n)\) for some \(n\geq 2\).
\end{proposition}
\begin{proof}
Suppose that \(S_c(W)\) is nilpotent.
Then it is \(n\)-nilpotent for some \(n\geq 2\), and consequently
every flat semiring in \(\mathsf{V}(S_c(W))\) is \(n\)-nilpotent.
It follows that \(\mathsf{V}(S_c(W))\) does not contain \(S_c(a_1\cdots a_n)\),
since the latter is not \(n\)-nilpotent.

Conversely, assume that \(S_c(W)\) is not nilpotent. Then \(W^\leq\) contains words of arbitrary length.
Hence for any \(n\geq 2\), there exists a word \(\mathbf{w}\) in \(W^\leq\) such that \(\ell(\mathbf{w})=n+1\).
By Corollary~\ref{wwi}, \(S_c(\mathbf{w}) \in \mathsf{V}(S_c(W))\).
By Lemmas~\ref{lemsc1} and~\ref{lemsc2}, \(S_c(a_1\cdots a_n)\in \mathsf{V}(S_c(\mathbf{w}))\).
Therefore, \(S_c(a_1\cdots a_n)\in \mathsf{V}(S_c(W))\).
\end{proof}

The following result generalizes Lemma~\ref{lemsc2} and further illustrates the special nature of $S(a^n)$.

\begin{lemma}\label{a1tam}
Let $\bw$ be a word in $X_c^+$. Then $S_c(\bw)$ belongs to $\mathsf{V}(S(a^n))$ for some $n\geq 1$.
\end{lemma}
\begin{proof}
We may write $\bw=a_1^{t_1}a_2^{t_2}\cdots a_m^{t_m}$,
where $a_1, \ldots, a_m$ are distinct letters and $t_1, \ldots, t_m$ are positive integers.
Let $n$ denote the number
\[
n = t_1 t_2 \cdots t_m + t_1 + t_2 + \cdots + t_m.
\]
We shall show that $S_c(\bw) \in \mathsf{V}(S(a^n))$.

Consider the direct product $(S(a^n))^{\times m}$ of $m$ copies of $S(a^n)$,
and let $A$ denote the subsemiring of $(S(a^n))^{\times m}$ generated by the elements
\[
\mathbf{a}_1 = (a^{k_1}, a, a, \ldots, a),\;
\mathbf{a}_2 = (a, a^{k_2}, a, \ldots, a),\;
\ldots,\;
\mathbf{a}_m = (a, a, \ldots, a, a^{k_m}),
\]
where $k_i = \frac{t_1 t_2 \cdots t_m}{t_i} + 1$ for each $1 \leq i \leq m$.
Let $\mathbf{a}$ denote the element $(a, a, \ldots, a)$ in $(S(a^n))^{\times m}$.
Then
\begin{align}\label{t1tn}
\mathbf{a}_1^{t_1} \mathbf{a}_2^{t_2} \cdots \mathbf{a}_m^{t_m} = \mathbf{a}^n.
\end{align}

Now let $J$ be the set of all elements of $A$ that do not divide $\mathbf{a}^n$ in $A$.
Since $J$ contains all tuples with a zero coordinate,
it follows that $J$ is both a multiplicative ideal and an additive order filter of $A$.
This allows us to consider the quotient semiring $A/J$,
which is a flat semiring consisting of the zero element $J$ and all singletons $\{\mathbf{b}\}$ for $\mathbf{b} \in A\setminus J$,
where each such $\mathbf{b}$ is of the form $\mathbf{a}_1^{x_1}\mathbf{a}_2^{x_2}\cdots\mathbf{a}_m^{x_m}$
with $x_i \geq 0$ for all $1 \leq i \leq m$.

Take an arbitrary element $\mathbf{a}_1^{x_1}\mathbf{a}_2^{x_2}\cdots\mathbf{a}_m^{x_m}$ of $A \setminus J$.
Then it divides $\mathbf{a}^n$ in $A$,
and so there exist nonnegative integers $y_1, \ldots, y_m$ such that
\[
\mathbf{a}_1^{x_1+y_1}\mathbf{a}_2^{x_2+y_2}\cdots\mathbf{a}_m^{x_m+y_m} = \mathbf{a}^n.
\]
Consequently, $(x_1+y_1, x_2+y_2, \ldots, x_m+y_m)$ is a solution of the system of linear equations
\begin{align}\label{lin-system}
\begin{pmatrix}
k_1 & 1  & \cdots & 1 \\
1 & k_2  & \cdots & 1 \\
\vdots & \vdots & \ddots & \vdots \\
1 & 1  & \cdots & k_m
\end{pmatrix}
\begin{pmatrix}
z_1 \\ z_2 \\ \vdots \\ z_m
\end{pmatrix}
=
\begin{pmatrix}
n \\ n \\ \vdots \\ n
\end{pmatrix}.
\end{align}
Let $M$ denote its coefficient matrix. Its determinant is
\[
\det(M)=\left(1 + \sum_{i=1}^m \frac{1}{k_i - 1}\right) \prod_{i=1}^m (k_i - 1)\neq 0.
\]
Hence the linear system~\eqref{lin-system} has at most one nonnegative integer solution.
By \eqref{t1tn}, \((t_1,\dots,t_m)\) is also a solution of
\eqref{lin-system}; hence \(x_i+y_i=t_i\) for all \(1\leq i\leq m\).
Hence
\[
A \setminus J = \{\mathbf{a}_1^{x_1} \cdots \mathbf{a}_m^{x_m} \mid 0 \leq x_i \leq t_i,\ 1 \leq i \leq m,\ (x_1,\dots,x_m) \neq (0,\dots,0)\}.
\]

Now suppose that two such products are equal:
\[
\mathbf{a}_1^{x_1}\mathbf{a}_2^{x_2} \cdots \mathbf{a}_m^{x_m}
=
\mathbf{a}_1^{y_1}\mathbf{a}_2^{y_2} \cdots \mathbf{a}_m^{y_m}
\]
with $0 \leq x_i, y_i \leq t_i$ for all $1 \leq i \leq m$.
Comparing exponents in each coordinate gives
\[
M(x_1 - y_1, \dots, x_m - y_m)^T = (0, \dots, 0)^T.
\]
Since $\det(M) \neq 0$, it follows that $x_i = y_i$ for all $1\leq i\leq m$.
Thus the representation of each element of $A \setminus J$ is unique.

With the above observations, it follows that $A/J$ is isomorphic to $S_c(a_1^{t_1} \cdots a_m^{t_m})$
under the mapping $\mathbf{a}_i/J \mapsto a_i$. Hence $S_c(a_1^{t_1} \cdots a_m^{t_m})$ belongs to $\mathsf{V}(S(a^n))$, as required.
\end{proof}

\begin{proposition}\label{scwa+}
Let $W$ be a nonempty set of words in $X^+_c$. Then $S_c(W)$ belongs to the variety $\mathsf{V}(S_c(\{a\}^+))$.
Consequently, $\mathsf{V}(S_c(\{a\}^+))$ is the maximum variety among those of the form $\mathsf{V}(S_c(W))$.
\end{proposition}

\begin{proof}
By Corollary~\ref{wwi}, we have that $\mathsf{V}(S_c(W)) = \mathsf{V}(S_c(\bw) \mid \bw \in W)$.
By Lemma~\ref{a1tam}, for each $\bw \in W$ there exists $n \geq 1$ such that $S_c(\bw) \in \mathsf{V}(S_c(a^n))$,
and so $\mathsf{V}(S_c(\bw) \mid \bw \in W)$ is a subvariety of $\mathsf{V}(S_c(a^n) \mid n \geq 1)$.
Finally, by Corollary~\ref{wwi} again, $\mathsf{V}(S_c(a^n) \mid n \geq 1) = \mathsf{V}(S_c(\{a\}^+))$.
Hence $\mathsf{V}(S_c(W))$ is a subvariety of $\mathsf{V}(S_c(\{a\}^+))$,
and so $S_c(W)$ belongs to $\mathsf{V}(S_c(\{a\}^+))$.
\end{proof}

\begin{corollary}\label{coro26061420}
Let $W$ be a nonempty set of words in $X^+_c$.
Then $\mathsf{V}(S_c(W))=\mathsf{V}(S_c(\{a\}^+))$ if and only if
$S_c(W)$ is not $k$-nil for any $k\geq 1$.
\end{corollary}
\begin{proof}
Suppose that $\mathsf{V}(S_c(W))=\mathsf{V}(S_c(\{a\}^+))$.
It follows immediately that $S_c(W)$ is not $k$-nil for any $k\geq 1$,
since $S_c(\{a\}^+)$ is not $k$-nil for any $k\geq 1$.

Conversely, assume that $S_c(W)$ is not $k$-nil for any $k\geq 1$.
By Proposition~\ref{scwa+}, $\mathsf{V}(S_c(W))$ is a subvariety of $\mathsf{V}(S_c(\{a\}^+))$,
so it remains to prove that $\mathsf{V}(S_c(\{a\}^+))$ is a subvariety of $\mathsf{V}(S_c(W))$.
Since $\mathsf{V}(S_c(\{a\}^+))=\mathsf{V}(S(a^k)\mid k\geq 1)$,
we need to prove that $S(a^k)$ belongs to $\mathsf{V}(S_c(W))$ for all $k\geq 1$.
Let $k\geq 1$ be an arbitrary integer.
Since $S_c(W)$ is not $k$-nil, there exists a word $\bw\in W^{\leq}$
such that $\bw^k\in W^{\leq}$.
By Corollary~\ref{wwi}, $S_c(\bw^k)$ lies in $\mathsf{V}(S_c(W))$.
Let $A$ denote the subsemiring of $S_c(\bw^k)$ generated by the single element $\bw$.
It is easy to see that $A$ is isomorphic to $S(a^k)$.
Hence $S(a^k)$ belongs to $\mathsf{V}(S_c(W))$, as required.
\end{proof}

\begin{corollary}
$\mathsf{V}(S_c(X_c^+))=\mathsf{V}(S_c(\{a\}^+))$.
\end{corollary}
\begin{proof}
It is easy to see that $S_c(X_c^+)$ is not $k$-nil for any $k\geq 1$.
The claim now follows directly from Corollary~\ref{coro26061420}.
\end{proof}

Two words $\mathbf{u}$ and $\mathbf{v}$ in $X_c^{+}$ are \emph{similar}
if they differ only in the names of letters; that is, $\mathbf{u}$ can be obtained from $\mathbf{v}$ by renaming the letters \cite{carpi2002}.
For example, $x_1x_2x_3^2$, $x_1x_2^2x_3$, and $x_2x_3x_4^2$ are all similar.
If $\bw_1$ and $\bw_2$ are words in $X_c^{+}$,
then it is easy to verify that $S_c(\bw_1)$ is isomorphic to $S_c(\bw_2)$ if and only if $\bw_1$ and $\bw_2$ are similar.
Combining this with Corollary~\ref{wwi}, we immediately deduce the following proposition.

\begin{proposition}\label{nilpotent}
Let $W$ be a nonempty set of words in $X^+_c$.
If $W'$ is a subset of $W$ such that every word in $W$ is similar to a subword of some word in $W'$,
then $\mathsf{V}(S_c(W))=\mathsf{V}(S_c(W'))$.
\end{proposition}

The following result is due to Wu et al.~\cite[Corollary 4.3]{wzr}.
\begin{lemma}\label{finite}
Let $W$ be a finite nonempty set of words in $X^+_c$.
Then $S_c(W)$ is finitely based if
and only if every word in $W$ is either a cube of a letter or a word of length at most $2$.
\end{lemma}

\begin{proposition}\label{prop:scwnilpotentfb}
Let $W$ be a nonempty set of words in $X^+_c$.
If $S_c(W)$ is a nilpotent flat semiring,
then $S_c(W)$ is finitely based if and only if
every word in $W$ is either a cube of a letter or a word of length at most $2$.
\end{proposition}
\begin{proof}
By assumption, $S_c(W)$ is $k$-nilpotent for some $k\geq 2$.
Then every word in $W$ has length less than $k$.
Consequently, up to similarity, there are only finitely many words in $W$.
Let $W'$ be a set consisting of exactly one representative from each similarity class.
Then $W'$ is a finite subset of $W$, and every word in $W$ is similar to a unique word in $W'$.
By Proposition~\ref{nilpotent}, $\mathsf{V}(S_c(W))=\mathsf{V}(S_c(W'))$,
and so $S_c(W)$ and $S_c(W')$ have the same finite basis property.
By Lemma~\ref{finite}, $S_c(W')$ is finitely based
if and only if every word in $W'$ is either a cube of a letter or a word of length at most $2$.
Therefore, $S_c(W)$ is finitely based
if and only if every word in $W$ is either a cube of a letter or a word of length at most $2$.
\end{proof}

\section{A sufficient condition for the nonfinitely based property}\label{sec:NFB}
In order to completely resolve the finite basis problem for flat semirings of the form $S_c(W)$,
Proposition~\ref{prop:scwnilpotentfb} tells us that it suffices to consider the case when $S_c(W)$ is not nilpotent.
In this section, we follow the hypergraph semiring approach developed by Jackson et al.~\cite{jrz} and Gao et al.~\cite{gjrz2},
and refine their sufficient conditions for the nonfinitely based property.
As applications, we show that certain non-nilpotent flat semirings $S_c(W)$ are nonfinitely based.
For the reader's convenience and completeness, we now recall some basic notions concerning hypergraph semirings.

Let \(k\geq 3\) be an integer.
A \emph{\(k\)-uniform hypergraph} \(\mathbb{H}\) is a pair \((V, E)\),
where \(E\) is a family of \(k\)-element subsets of a set \(V\).
Each element of \(V\) is a \emph{vertex} of \(\mathbb{H}\), and each
element of \(E\) is a \emph{hyperedge} of \(\mathbb{H}\).
Throughout, we assume that hypergraphs are finite and have no isolated
vertices, that is, every vertex is contained in at least one hyperedge.

Let $\mathbb{H}=(V, E)$ be a $k$-uniform hypergraph.
Then $\mathbb{H}$ is \emph{$2$-colourable} if there exists a mapping
$\varphi\colon V \to \{0, 1\}$ such that for every hyperedge $e \in E$,
the image $\varphi(e)$ contains both $0$ and $1$; that is, $|\varphi(e)|=2$.
A \emph{cycle} of length $n\geq 2$ in $\mathbb{H}$ is an alternating sequence $v_1, e_1, v_2, e_2, \ldots, v_n, e_n$ of
distinct vertices and hyperedges such that $v_1 \in e_1 \cap e_n$ and $v_{i+1} \in e_i \cap e_{i+1}$ for $1 \leq i < n$.
The \emph{girth} of \(\mathbb{H}\) is the length of its shortest cycle,
with the convention that the girth is \(\infty\) if \(\mathbb{H}\) has no
cycle. Accordingly, a \emph{hyperforest} is a hypergraph of girth
\(\infty\), that is, a hypergraph without cycles.

Let $\mathbb{H} = (V, E)$ be a $k$-uniform hypergraph  with girth at least $4$.
A \emph{hypergraph semiring} $S_\mathbb{H}$ defined by $\mathbb{H}$ is a flat semiring generated by a copy
$\{\mathbf{a}_v \mid v\in V\}$ of $V$ along with a fresh symbol~$0$, subject to the following rules:
\begin{enumerate}[label=(\arabic*), font=\normalfont]
\item $0$ is the multiplicative zero element;

\item $\mathbf{a}_u\mathbf{a}_v=\mathbf{a}_v\mathbf{a}_u$ for all $u, v\in V$;

\item $\mathbf{a}_u\mathbf{a}_v=0$ if $\{u, v\}$ is not a $2$-element subset of a hyperedge in $E$;

\item\label{item4} $\mathbf{a}_{u_1}\cdots\mathbf{a}_{u_k}=\mathbf{a}_{v_1}\cdots\mathbf{a}_{v_k}$
if $\{u_1, \ldots, u_k\}, \{v_1, \ldots, v_k\}\in E$;

\item $\mathbf{a}_{u_1}\cdots\mathbf{a}_{u_{k-1}}=\mathbf{a}_{v_1}\cdots\mathbf{a}_{v_{k-1}}$
if $\{u_1, \ldots, u_{k-1}, v\}, \{v_1, \ldots, v_{k-1}, v\} \in E$ for some $v \in V$.
\end{enumerate}
We let $\mathbf{a}$ denote the common value of the products in \ref{item4}.

For each integer \(k\geq 3\), we fix, for every \(n\geq 1\), a
\(k\)-uniform hypergraph \(\mathbb{H}_n=(V_n,E_n)\) that is not
\(2\)-colourable and whose girth is greater than \(k\binom{kn}{2}\).
Moreover, $\mathbb{H}_n$ has no isolated vertex; that is, every vertex of $\mathbb{H}_n$ lies in some hyperedge.
The existence of such hypergraphs is guaranteed by \cite[Theorem 2.7]{ham18}.

Let $\{x_v \mid v \in V_n\}$ be a set of letters in one-to-one correspondence with $V_n$.
We denote by $\mathbf{t}_{\mathbb{H}_n}$ the term
\[
\sum_{\{v_1,v_2,\ldots,v_k\} \in E_n} x_{v_1} x_{v_2} \cdots x_{v_k}.
\]
Since hyperedges are unordered,
each hyperedge gives rise to $k!$ distinct hyperedge products: for every hyperedge $\{v_1, \dots, v_k\} \in E_n$
and every permutation $\sigma$ of $\{1,\dots,k\}$, the word $x_{v_{\sigma(1)}} \cdots x_{v_{\sigma(k)}}$
appears in $\mathbf{t}_{\mathbb{H}_n}$.

We are now ready to state the main result of this section,
which provides a sufficient condition for an ai-semiring variety to be nonfinitely based.
\begin{theorem}\label{thnfb}
Let \(\mathcal{V}\) be an ai-semiring variety that contains the flat semiring \(S_c(a_1 a_2 \cdots a_k)\) for some \(k \geq 3\).
If \(\mathcal{V}\) satisfies the identities
\begin{equation}\label{eqthth2}
\mathbf{t}_{\mathbb{H}_n} \approx \mathbf{t}_{\mathbb{H}_n}^2
\end{equation}
for all \(n \geq 1\),
then \(\mathcal{V}\) is nonfinitely based.
\end{theorem}

\begin{proof}
The proof proceeds as follows.
We show that for every \(n \geq 1\), the set of all \(n\)-variable identities of \(\mathcal{V}\) fails to form a basis for the equational theory of \(\mathcal{V}\).
To this end, it suffices to prove that for every \(n \geq 1\),
the hypergraph semiring \(S_{\mathbb{H}_n}\) does not belong to \(\mathcal{V}\),
whereas every \(n\)-generated subalgebra of \(S_{\mathbb{H}_n}\) does belong to \(\mathcal{V}\).

Let \(n \geq 1\).
By the proof of \cite[Theorem 4.9]{jrz},
every \(n\)-generated subalgebra \(T\) of \(S_{\mathbb{H}_n}\) is a subalgebra of a hypergraph semiring \(S_{\mathbb{G}^+}\),
where \(\mathbb{G}^+\) is a \(k\)-uniform subhypergraph of \(\mathbb{H}_n\) with at most \(k\binom{kn}{2}\) vertices.
It follows that \(\mathbb{G}^+\) contains no cycles, and hence is a hyperforest.
By \cite[Lemma 4.2]{jrz}, we have that \(S_{\mathbb{G}^+}\) lies in the variety \(\mathsf{V}(S_c(a_1 \cdots a_k))\).
Since \(S_c(a_1 \cdots a_k)\) lies in \(\mathcal{V}\),
we conclude that every \(n\)-generated subalgebra \(T\) of \(S_{\mathbb{H}_n}\) belongs to $\mathcal{V}$.

Now consider the assignment \(\varphi\colon \{x_v \mid v \in V_n\} \to S_{\mathbb{H}_n}\)
defined by \(\varphi(x_v) = \mathbf{a}_v\) for all \(v \in V_n\).
It is readily seen that \(\varphi(\mathbf{t}_{\mathbb{H}_n}) = \mathbf{a}\), and hence \(\varphi(\mathbf{t}_{\mathbb{H}_n}^2) = \mathbf{a}^2 = 0\).
Thus \(S_{\mathbb{H}_n}\) does not satisfy identity~\eqref{eqthth2}.
Consequently, \(S_{\mathbb{H}_n}\) is not a member of \(\mathcal{V}\).
Therefore, \(\mathcal{V}\) is nonfinitely based, as required.
\end{proof}

Let $\mathcal{V}_1$ and $\mathcal{V}_2$ be ai-semiring varieties such that
$\mathcal{V}_1$ is a subvariety of $\mathcal{V}_2$. Then the interval
$[\mathcal{V}_1, \mathcal{V}_2]$ denotes the set of all subvarieties of $\mathcal{V}_2$
that contain $\mathcal{V}_1$.

\begin{proposition}\label{prop:scwnfbknil}
Let \(W\) be a nonempty set of words in \(X^+_c\). Suppose that
\(S_c(W)\) is not nilpotent but is \(k\)-nil for some \(k\geq 3\).
Then every variety in the interval
\[
[\mathsf{V}(S_c(a_1\cdots a_k)),\mathsf{V}(S_c(W))]
\]
is nonfinitely based. In particular, \(S_c(W)\) itself is
nonfinitely based.
\end{proposition}
\begin{proof}
Since $S_c(W)$ is not nilpotent,
it follows from Proposition~\ref{nilpotentscw} that $\mathsf{V}(S_c(W))$ contains $S_c(a_1\cdots a_k)$,
and so  $[\mathsf{V}(S_c(a_1\cdots a_k)),\mathsf{V}(S_c(W))]$ is well-defined.

Now let $\mathcal{V}$ be an arbitrary variety in this interval.
To show that $\mathcal{V}$ is nonfinitely based, by Theorem~\ref{thnfb} it suffices to show that
$S_c(W)$ satisfies the identities~\eqref{eqthth2} for each $n\geq 1$.

Let $\psi \colon \{x_v \mid v \in V_n\} \to S_c(W)$ be an arbitrary assignment.
If $\psi(\bt_{\mathbb{H}_n}) = 0$, then
\[
\psi(\bt_{\mathbb{H}_n})=0=\psi(\bt_{\mathbb{H}_n}^2).
\]
Now suppose that $\psi(\bt_{\mathbb{H}_n}) \neq 0$.
Then there exists a word $\bw \in W^\leq$ such that
\[
\psi(x_{v_1}) \psi(x_{v_2}) \cdots \psi(x_{v_k}) = \bw
\]
for each hyperedge $\{v_1, \ldots, v_k\} \in E_n$.
We may write $x_1^{t_1} \cdots x_s^{t_s}$ for $\bw$,
where $x_1, \ldots, x_s$ are distinct variables, and
$t_1, \ldots, t_s$ are positive integers.
By assumption, $S_c(W)$ is $k$-nil; hence $t_1 < k$.
Consequently, for each hyperedge $\{v_1, \ldots, v_k\} \in E_n$,
there exist indices $1 \leq i, j \leq k$ such that $x_1$ occurs in $\psi(x_{v_i})$,
but does not occur in $\psi(x_{v_j})$.
Define a mapping $\chi \colon V(\mathbb H_n) \to \{0, 1\}$ by
\[
\chi(v) =
\begin{cases}
0, & \text{if } x_1 \text{ occurs in } \psi(x_v), \\
1, & \text{otherwise}.
\end{cases}
\]
Then $\chi$ is a $2$-colouring of $\mathbb H_n$, contradicting the fact that $\mathbb H_n$ is not $2$-colourable.
So $\psi(\bt_{\mathbb{H}_n}) \neq 0$ is impossible.

Therefore, $S_c(W)$ satisfies the identities \eqref{eqthth2}, and the desired result follows.
\end{proof}

A \emph{primitive word} in \(X^*\) is a word that is not a power of any other word.
By Propositions~1.3.1 and~1.3.2 in~\cite{loth}, we have the following lemma.

\begin{lemma}\label{primitive}
Every nonempty word in \(X^*\) can be expressed uniquely as a power of a primitive word.
Moreover, a set of words is pairwise commuting if and only if its elements can be expressed as powers of one and the same primitive word.
\end{lemma}

\begin{proposition}\label{mcwmw}
Let \(S\) be a flat semiring of the form \(M_c(W)\) or \(M(W)\),
where \(W\) contains a nonempty word.
If \(S\) satisfies the identity \(x^{k}\approx x^{k+1}\) for some \(k\geq 3\),
then every variety in the interval
\([\mathsf{V}(S_c(a_1\cdots a_k)),\mathsf{V}(S)]\) is nonfinitely based.
In particular, \(S\) itself is nonfinitely based.
\end{proposition}

\begin{proof}
By \cite[Proposition 2.6]{jrz}, \(S_c(a_1\cdots a_k)\in \mathsf{V}(M(a))\).
Now take any nonempty word \(\mathbf{w}\in W\), and consider the
subsemiring \(T\) of $S$ generated by \(\{1,\mathbf{w}\}\).
Let \(J=\{\mathbf{w}^{i}\mid i\geq 2\}\cup\{0\}\).
It is easy to verify that \(J\) is an ideal of \(T\) and that
\(T/J\cong M(a)\).
Hence \(M(a)\in \mathsf{V}(S)\).
This implies that
\(S_c(a_1\cdots a_k)\in \mathsf{V}(S)\) for each \(k\geq 3\).
Thus the interval
\([\mathsf{V}(S_c(a_1\cdots a_k)),\mathsf{V}(S)]\) is well-defined.

Now let \(\mathcal{V}\) be an arbitrary variety in this interval.
To show that \(\mathcal{V}\) is nonfinitely based, by Theorem~\ref{thnfb} it suffices to show that
\(S\) satisfies the identities~\eqref{eqthth2} for every \(n\geq 1\).
Let \(\psi \colon \{x_v \mid v \in V_n\} \to S\) be an arbitrary assignment.
If \(\psi(\mathbf{t}_{\mathbb{H}_n}) \in \{0,1\}\), then
\[
\psi(\mathbf{t}_{\mathbb{H}_n}) = \psi(\mathbf{t}_{\mathbb{H}_n}^2).
\]
Now suppose that \(\psi(\mathbf{t}_{\mathbb{H}_n}) \neq 0,1\).
Then there exists a word \(\mathbf{w} \in W^{\leq}\setminus \{1\}\) such that
\[
\psi(x_{v_1}) \psi(x_{v_2}) \cdots \psi(x_{v_k}) = \mathbf{w}
\]
for each hyperedge \(\{v_1, \ldots, v_k\} \in E_n\).
Consider the following two cases.

\textbf{Case 1.} \(S=M_c(W)\).
Write \(\mathbf{w}=x_1^{t_1}\cdots x_s^{t_s}\).
Since \(S\) satisfies \(x^k\approx x^{k+1}\),
it follows that \(1\leq t_1< k\).
Consequently, for each hyperedge \(\{v_1, \ldots, v_k\} \in E_n\),
there exist indices \(1 \leq i, j \leq k\) such that \(x_1\) occurs in \(\psi(x_{v_i})\)
but does not occur in \(\psi(x_{v_j})\).
Define a mapping \(\chi \colon V(\mathbb{H}_n) \to \{0, 1\}\) by
\[
\chi(v) =
\begin{cases}
0, & \text{if } x_1 \text{ occurs in } \psi(x_v), \\
1, & \text{otherwise}.
\end{cases}
\]
Then \(\chi\) is a \(2\)-colouring of \(\mathbb{H}_n\), contradicting the fact that \(\mathbb{H}_n\) is not \(2\)-colourable.

\textbf{Case 2.} \(S=M(W)\).
Let \(\{v_1,\ldots,v_k\}\) be an arbitrary hyperedge.
Since
\[
\psi(x_{v_i})\psi(x_{v_j}) \prod_{l\neq i,j}\psi(x_{v_l})
=
\psi(x_{v_j})\psi(x_{v_i}) \prod_{l\neq i,j}\psi(x_{v_l})
= \mathbf{w}\neq 0,
\]
it follows that
\[
\psi(x_{v_i})\psi(x_{v_j}) = \psi(x_{v_j})\psi(x_{v_i}).
\]
Thus \(\psi(x_{v_1}),\ldots,\psi(x_{v_k})\) are pairwise commuting words in \(X^*\).
By Lemma~\ref{primitive}, there exists a primitive word \(\mathbf{p}\)
such that \(\psi(x_{v_i}) = \mathbf{p}^{t_i}\) for some \(t_i\geq 0\).
Since \(S\) satisfies the identity \(x^k\approx x^{k+1}\) and since \(\mathbf{w}\neq 1\),
it follows that some \(t_i>0\) and some \(t_j=0\); that is, some \(\psi(x_v)\) is an empty word, while some is not.
Since \(\{v_1,\ldots,v_k\}\) is the arbitrary hyperedge, we may define
a mapping \(\chi \colon V(\mathbb{H}_n) \to \{0, 1\}\) by
\[
\chi(v) =
\begin{cases}
0, & \text{if } \psi(x_v) \text{ is an empty word}, \\
1, & \text{otherwise}.
\end{cases}
\]
Then \(\chi\) is a \(2\)-colouring of \(\mathbb{H}_n\), contradicting the fact that \(\mathbb{H}_n\) is not \(2\)-colourable.

In conclusion, \(S\) satisfies the identities~\eqref{eqthth2}, and hence the desired conclusion holds.
\end{proof}

\section{The flat semiring $\flat(\mathbb{N})$}\label{sec:flatn}
Let \(\mathbb{N}\) denote the set of all positive integers and
\(\mathbb{P}\) the set of all prime numbers.
It is easy to see that \(S_c(\{a\}^+)\) is isomorphic to the flat
extension \(\flat(\mathbb{N})\) of the additive semigroup
\(\mathbb{N}\).
To completely solve the finite basis problem for the flat semirings
\(S_c(W)\), by Propositions~\ref{prop:scwnilpotentfb} and~\ref{prop:scwnfbknil}, it remains to consider the case where
\(S_c(W)\) is not \(k\)-nil for any \(k \geq 3\).
By Corollary~\ref{coro26061420}, this is precisely the case where
\[
\mathsf{V}(S_c(W)) = \mathsf{V}(S_c(\{a\}^+)) = \mathsf{V}(\flat(\mathbb{N})).
\]
Consequently, the finite basis problem for flat semirings of the form
\(S_c(W)\) reduces to that for \(\flat(\mathbb{N})\).

The strategy for proving that \(\flat(\mathbb{N})\) is nonfinitely
based is as follows. For each prime \(p\), we define an associated
identity \(\sigma_p\) and algebra \(A_p\). First, we show that every
identity \(\sigma_p\) holds in \(\flat(\mathbb{N})\) but fails in
\(A_p\). Second, we prove that for any finite set \(\Sigma\) of
identities valid in \(\flat(\mathbb{N})\), there exists a prime \(p\)
such that \(A_p\) satisfies all identities in \(\Sigma\). This
establishes that no finite set \(\Sigma\) of identities valid in
\(\flat(\mathbb{N})\) is an equational basis for
\(\flat(\mathbb{N})\), and hence that \(\flat(\mathbb{N})\) is indeed
nonfinitely based. Using this idea, we obtain a more general result: all varieties
lying between two designated varieties are nonfinitely based.

We now turn to the construction. For each \(p \in \mathbb{P}\), let
\(\sigma_p\) denote the identity
\[
\sigma_p\colon (x+y)^p\approx x^p+y^p.
\]

The law \(\sigma_p\) is a particularly simple and familiar identity.
In number-theoretic contexts, where \(+\) is ordinary addition on
\(\mathbb{N}\), it is often called the \emph{poor student's Binomial Theorem}
or the \emph{Freshman's Binomial Theorem}, because it is exactly the
kind of erroneous simplification sometimes made by a weak student.
Nonetheless, it is a genuinely useful property: it does hold in the
ring \(\mathbb{Z}_p\) of integers modulo \(p\), since in the true binomial
expansion
\[
(x+y)^p=\sum_{k=0}^{p}\binom{p}{k}x^k y^{p-k},
\]
all nontrivial binomial coefficients \(\binom{p}{k}\), for
\(1\leq k\leq p-1\), are divisible by \(p\), and hence vanish modulo
\(p\). Thus the apparent ``mistake'' becomes a valid identity in that
setting. In fact, for any primes \(p\) and \(q\), the ring
\(\mathbb{Z}_q\) satisfies the identity \(\sigma_p\) if and only if
\(p\equiv 1 \pmod{q-1}\).

We next observe that the identity \(\sigma_p\) has a natural interpretation
in the max-plus semiring \((\mathbb Z,\max,+)\).
For any \(x,y\in\mathbb Z\), and for each prime \(p\), the identity
\(\sigma_p\) translates in the max-plus semiring \((\mathbb Z,\max,+)\)
to
\[
p\cdot \max\{x,y\} = \max\{px,py\}.
\]
This is plainly true: if \(x\geq y\), then \(px\geq py\), and hence both
sides equal \(px\); the case \(y\geq x\) is symmetric. Thus we obtain the
following conclusion.

\begin{lemma}\label{lem:max+r}
The max-plus algebra \((\mathbb Z,\max,+)\) satisfies every identity
\(\sigma_p\).
\end{lemma}

Of course, the main object of study in this section is the flat
semiring. We can prove the following lemma.

\begin{lemma}\label{lem:flatsigmap}
A flat semiring \(S\) satisfies the identity \(\sigma_p\) if and only
if \(a^p+b^p=0\) whenever \(a,b\in S\) with
\(a\neq b\).
\end{lemma}

\begin{proof}
Suppose first that \(S\) satisfies \(\sigma_p\), and let \(a,b\in S\)
with \(a\neq b\). Then \(a+b=0\), so
\[
a^p+b^p=(a+b)^p=0^p=0,
\]
as required.

Conversely, suppose that for all \(a,b\in S\), \(a\neq b\) implies
\(a^p+b^p=0\). We must show that \((a+b)^p=a^p+b^p\) for all
\(a,b\in S\). If \(a=b\), then
\[
(a+b)^p=a^p=a^p+b^p,
\]
using idempotence of addition. If \(a\neq b\), then \(a+b=0\), so
\[
(a+b)^p=0^p=0=a^p+b^p
\]
by the assumed condition. Thus \(\sigma_p\) holds in \(S\).
\end{proof}

\begin{corollary}\label{coro:4ssigmap}
Every flat semiring of one of the forms \(S(W)\), \(S_c(W)\),
\(M(W)\) and \(M_c(W)\) satisfies the identity \(\sigma_p\) for each
prime \(p\).
\end{corollary}
\begin{proof}
Let \(S\) be a flat semiring of one of the forms \(S(W)\),
\(S_c(W)\), \(M(W)\) and \(M_c(W)\). Suppose that \(a,b\in S\) with
\(a\neq b\). We claim that \(a^p+b^p=0\). Assume otherwise, so that
\(a^p+b^p\neq 0\). Then \(a^p=b^p\neq 0\). Since any two words in
either the free monoid or the free commutative monoid whose \(p\)-th
powers coincide must themselves be equal, we obtain \(a=b\), a
contradiction. Hence \(a\neq b\) implies \(a^p+b^p=0\). By
Lemma~\ref{lem:flatsigmap}, \(S\) satisfies the identity \(\sigma_p\).
\end{proof}

Now let \(\mathcal{W}_{\mathbb{P}}\) be the ai-semiring variety
defined by the set \(\{\sigma_p \mid p\in \mathbb{P}\}\). By
Corollary~\ref{coro:4ssigmap}, \(\mathsf{V}(\flat(\mathbb{N}))\) is a subvariety
of \(\mathcal{W}_{\mathbb{P}}\).

For an ai-semiring variety \(\mathcal{V}\) and an integer \(n\geq 1\),
let \(\mathcal{V}^{(n)}\) denote the variety defined by all
\(n\)-variable identities of \(\mathcal{V}\), where an
\emph{\(n\)-variable identity} means an identity
\(\mathbf{u}\approx\mathbf{v}\) with
\(|c(\mathbf{u})\cup c(\mathbf{v})|\leq n\).
Then
\[
\mathcal{V}^{(1)} \supseteq \mathcal{V}^{(2)} \supseteq \cdots \supseteq
\mathcal{V}^{(n)} \supseteq \mathcal{V}^{(n+1)} \supseteq\cdots\supseteq \mathcal{V}.
\]
Moreover, if $\mathcal{V}$ is finitely based,
then there exists $n \geq 1$ such that $\mathcal{V}=\mathcal{V}^{(n)}$,
and so $\mathcal{V}^{(n)}$ is also finitely based.
Equivalently, if every $\mathcal{V}^{(n)}$ is nonfinitely based for sufficiently large $n$,
then $\mathcal{V}$ itself is also nonfinitely based.

Next, we introduce the main
theorem of this section.

\begin{theorem}\label{mainthm}
Let $\mathcal{V}$ be a variety in the interval $[\mathsf{V}(\flat(\mathbb{N})), \mathcal{W}_{\mathbb{P}}]$.
Then the variety $\mathcal{V}^{(n)}$ is nonfinitely based for each $n\geq 2$.
In particular, $\mathcal{V}$ itself is nonfinitely based.
\end{theorem}

Next, we introduce the construction of the algebra \(A_p\), which is
the combination of the two familiar flat semirings
\(S(a^p)\) and \(\flat(\mathbb{Z}_p)\). Specifically, for each
\(p \in \mathbb{P}\), set
\[
A_p := (\{1,\dots,p\} \times \mathbb{Z}_p) \cup \{0\},
\]
where \(\mathbb{Z}_p\) is the cyclic group of integers modulo \(p\).
Define a multiplication \(\cdot\) on \(A_p\) by
\[
(m,\bar g) \cdot (n,\bar h) =
\begin{cases}
(m+n,\overline{g+h}), & m+n \leq p, \\
0, & m+n > p,
\end{cases}
\]
and
\[
(m,\bar g)\cdot 0 = 0\cdot (m,\bar g) = 0\cdot 0 = 0.
\]
It is straightforward to verify that \((A_p,\cdot)\) is a
\(0\)-cancellative commutative semigroup. Hence \(A_p\) becomes a
flat semiring.

The following lemma describes the precise relationship between the
algebra \(A_p\) and the identities \(\sigma_q\).
\begin{lemma}\label{lem:separation}
Let $p, q \in \mathbb{P}$. Then the flat semiring $A_p$ satisfies $\sigma_q$ if and only if $p\neq q$.
\end{lemma}
\begin{proof}
First, we show that \(A_p\) does not satisfy \(\sigma_p\). Take
$x=(1,\bar 0)$ and $y=(1,\bar 1)$. Then $x+y=0$, and so $(x+y)^p = 0$, whereas
\[
x^p+y^p
=(1,\bar 0)^p+(1,\bar 1)^p
=(p,\bar 0)+(p,\bar 0)=(p,\bar 0).
\]
Thus \(A_p\) does not satisfy \(\sigma_p\). Consequently, if \(A_p\)
satisfies \(\sigma_q\), then \(q\neq p\).

Conversely, assume that \(p\neq q\). We prove that \(A_p\) satisfies
\(\sigma_q\). We claim
that \(\alpha^q+\beta^q=0\) whenever \(\alpha,\beta\in A_p\) with
\(\alpha\neq \beta\). Suppose, to the contrary, that
\(\alpha^q+\beta^q\neq 0\). Then \(\alpha^q=\beta^q\neq 0\). Write
\(\alpha=(m_1,\overline{g_1})\) and
\(\beta=(m_2,\overline{g_2})\). Then
\[
(qm_1,\overline{qg_1})=(qm_2,\overline{qg_2})\neq 0.
\]
Thus \(qm_1=qm_2\leq p\) and \(\overline{qg_1}=\overline{qg_2}\).
Consequently, \(m_1=m_2\) and \(p\mid q(g_1-g_2)\). Since \(p\) and
\(q\) are distinct primes, they are coprime, so \(p\mid (g_1-g_2)\),
and hence \(\overline{g_1}=\overline{g_2}\). Therefore
\(\alpha=\beta\), contradicting \(\alpha\neq\beta\). Thus
\(\alpha^q+\beta^q=0\) whenever \(\alpha\neq\beta\). By
Lemma~\ref{lem:flatsigmap}, \(A_p\) satisfies \(\sigma_q\), as
required.
\end{proof}

\begin{remark}
Similarly, one can prove that for any primes \(p\) and \(q\), the
flat semiring \(\flat(\mathbb{Z}_p)\) satisfies the identity
\(\sigma_q\) if and only if \(p\neq q\). Thus, using identities
that do not involve the constant \(1\), we recover
\cite[Proposition~3.1]{mssemifield}. This also shows that the properties of
\(\flat(\mathbb{Z}_p)\) and \(\mathbb{Z}_p\) are very different.
\end{remark}

\begin{proposition}\label{propapp}
The variety $\mathsf{V}(\{A_p\mid p\in\mathbb{P}\})$ has continuum many subvarieties.
\end{proposition}
\begin{proof}
Let $\mathcal{V} = \mathsf{V}(\{A_p\mid p\in\mathbb{P}\})$.
For each subset $Q$ of $\mathbb{P}$, let $\mathcal{V}_Q$ denote the variety generated by $\{A_p\mid p\in Q\}$.
For distinct subsets $Q_1, Q_2$ of $\mathbb{P}$, suppose without loss of generality that $p \in Q_2 \setminus Q_1$.
By Lemma~\ref{lem:separation}, $\mathcal{V}_{Q_1}$ satisfies $\sigma_p$ but $\mathcal{V}_{Q_2}$ does not.
Hence $\mathcal{V}_{Q_1}\neq\mathcal{V}_{Q_2}$.
Thus distinct subsets of $\mathbb{P}$ give rise to distinct subvarieties of $\mathcal{V}$.
Since $\mathbb{P}$ has continuum many subsets, $\mathcal{V}$ has continuum many subvarieties.
\end{proof}

Dually, we have the following proposition.

\begin{proposition}
There are continuum many ai-semiring varieties containing $\mathcal{W}_{\mathbb{P}}$.
\end{proposition}
\begin{proof}
The proof is analogous to that of Proposition~\ref{propapp}, but with the roles of algebras and identities reversed.

For each subset $Q$ of $\mathbb{P}$, let $\mathcal{W}_Q$ denote the ai-semiring variety defined by $\{\sigma_p\mid p\in Q\}$.
For distinct subsets $Q_1, Q_2$ of $\mathbb{P}$, suppose without loss of generality that $p \in Q_2 \setminus Q_1$.
By Lemma~\ref{lem:separation}, $A_p$ lies in $\mathcal{W}_{Q_1}$, but not in $\mathcal{W}_{Q_2}$.
Hence $\mathcal{W}_{Q_1}\neq\mathcal{W}_{Q_2}$.
Thus distinct subsets of $\mathbb{P}$ give rise to distinct ai-semiring varieties.
Since $\mathbb{P}$ has continuum many subsets,
there are continuum many ai-semiring varieties containing $\mathcal{W}_{\mathbb{P}}$.
\end{proof}

Next, we characterize the identities holding in $\flat(\mathbb{N})$ and in $A_p$. To this end, we introduce the following conventions and notation.

We regard vectors as ordered tuples and do not distinguish
notationally between row and column vectors; the intended
interpretation is determined by context. For instance, if
\(\mathbf{D}\) is an \(m\times n\) matrix and \(\mathbf{a},\mathbf{b}\)
are vectors, then the product \(\mathbf{D}\mathbf{a}\) indicates that
\(\mathbf{a}\) is an \(n\)-dimensional column vector, while the
product \(\mathbf{b}\mathbf{D}\) indicates that \(\mathbf{b}\) is an
\(m\)-dimensional row vector. The product \(\mathbf{a}\mathbf{b}\) is
understood as the standard inner product, with \(\mathbf{a}\) a row
vector and \(\mathbf{b}\) a column vector of the same dimension.

Let \(\mathbf{u}=\mathbf{w}_1+\cdots+\mathbf{w}_m\) be an ai-semiring
term with \(c(\mathbf{u})\subseteq\{x_1,\dots,x_n\}\).
When considering the evaluation of \(\mathbf{u}\) in commutative
semirings, such as \(\flat(\mathbb{N})\) and \(A_p\), we need to
consider the canonical form of each \(\mathbf{w}_i\) in the commutative
sense, namely
\[
x_1^{\alpha_{i1}}x_2^{\alpha_{i2}}\cdots x_n^{\alpha_{in}},
\]
where \(\alpha_{ik}=\occ(x_k,\mathbf{w}_i)\geq 0\) for each \(i\).
Denote by
\[
\boldsymbol{\alpha}_i=(\alpha_{i1},\alpha_{i2},\ldots,\alpha_{in})
\]
the exponent vector of \(\mathbf{w}_i\).
Now, for \(m\geq 2\), define the difference matrix
\(\mathbf{D}_{\mathbf{u}}\) by
\[
\mathbf{D}_{\mathbf{u}}=
\begin{pmatrix}
\boldsymbol{\alpha}_2-\boldsymbol{\alpha}_1\\
\boldsymbol{\alpha}_3-\boldsymbol{\alpha}_1\\
\vdots\\
\boldsymbol{\alpha}_m-\boldsymbol{\alpha}_1
\end{pmatrix},
\]
and define it to be the \(1\times n\) zero matrix for \(m=1\).
Also, let
\[
\ker_{\mathbb{N}}\mathbf{D}_{\mathbf{u}}
=
\{\mathbf{d}\in\mathbb{N}^n\mid
\mathbf{D}_{\mathbf{u}}\mathbf{d}=\mathbf{0}\}.
\]
\begin{remark}
Here we define \(\mathbf{D}_{\mathbf{u}}\) over
\(\{x_1,\ldots,x_n\}\) rather than over \(c(\mathbf{u})\).
The reason is that in applications we need to discuss two terms
\(\mathbf{u}\) and \(\mathbf{v}\), and hence we need to work over
\(c(\mathbf{u})\cup c(\mathbf{v})\).
\end{remark}
\begin{lemma}\label{lem:evaluation}
Let $\mathbf{u}=\mathbf{w}_1+\cdots+\mathbf{w}_m$ be an ai-semiring term with $c(\mathbf{u})\subseteq \{x_1,\ldots,x_n\}$, and let $\mathbf{d}=(d_1,\ldots,d_n),\mathbf{g}=(g_1,\ldots,g_n)\in\mathbb N^n$, where $m,n\geq1$.
\begin{enumerate}[label=$(\arabic*)$]
\item \label{item 11}
In either $\flat(\mathbb{N})$ or $A_p$, under any evaluation,
if a variable in $c(\mathbf{u})$ is assigned $0$,
then $\mathbf{u}$ evaluates to $0$.

\item \label{item 12} Under the evaluation $\varphi(x_j)=a^{d_j}$ in $\flat(\mathbb{N})$,
$\varphi(\mathbf{u})\neq0$ if and only if $\D_{\bu}\bd = \bzero$.
In this case, $\varphi(\mathbf{u})=a^{\boldsymbol{\alpha}_i\mathbf{d}}$ for each $1\leq i\leq m$.

\item \label{item 13}
Under the evaluation $\psi(x_j)=(d_j,\bar g_j)$ in $A_p$ with $d_j\leq p$,
$\psi(\mathbf{u})\neq0$ if and only if
$\boldsymbol{\alpha}_i\mathbf{d}\leq p$ for every $1\leq i\leq m$,
$\D_{\bu}\bd = \bzero$, and
$\D_{\bu}\bg\equiv\bzero\pmod p$,
where the congruence is understood coordinatewise. In this case,
$\psi(\mathbf{u})=(\boldsymbol{\alpha}_i\mathbf{d},
\overline{\boldsymbol{\alpha}_i\mathbf{g}})$ for each
\(1\leq i\leq m\).
\end{enumerate}
\end{lemma}

\begin{proof}
For \ref{item 11},
since $0$ is absorbing for addition and multiplication in both algebras, assigning a variable in $c(\bu)$ to $0$ forces the whole term $\bu$ to evaluate to $0$.

For \ref{item 12} and \ref{item 13}, we first consider the simple case \(m=1\). In $\flat(\mathbb N)$,
\[
\varphi(\mathbf{u}) = \varphi(\mathbf{w}_1)
= (a^{d_1})^{\alpha_{11}}(a^{d_2})^{\alpha_{12}}
\cdots (a^{d_n})^{\alpha_{1n}}
= a^{d_1\alpha_{11}+\cdots+d_n\alpha_{1n}}
= a^{\balpha_1\bd},
\]
where factors with exponent zero are omitted. Thus in $\flat(\mathbb{N})$, $\varphi(\bu) = a^{\balpha_1\bd} \neq 0$ and $\D_{\bu} = \bzero$ for every $\mathbf{d}$, since $\mathbf{D}_{\mathbf{u}}$ is the $1\times n$ zero matrix.
Similarly, in $A_p$
\[
\psi(\bu) = \psi(\bw_1) =
\begin{cases}
 (\balpha_1\bd, \overline{\balpha_1\bg}),& \balpha_1\bd\leq p,\\
 0,& \balpha_1\bd> p.
\end{cases}
\]
Thus $\psi(\bu)\neq 0$ precisely when $\boldsymbol{\alpha}_1\mathbf{d}\leq p$. Also $\mathbf{D}_{\mathbf{u}}\mathbf{d}=0$ and $\mathbf{D}_{\mathbf{u}}\mathbf{g}\equiv0\pmod p$ for every $\mathbf{d},\mathbf{g}\in \mathbb N^n$.
Thus assume $m\geq2$ for the remainder of the proof.

For \ref{item 12}, under $\varphi(x_j)=a^{d_j}$ in $\flat(\mathbb{N})$, the word $\mathbf{w}_i$ evaluates to $a^{\boldsymbol{\alpha}_i\mathbf{d}}$. Since addition is flat, $\varphi(\mathbf{u})$ is nonzero if and only if all $\boldsymbol{\alpha}_i\mathbf{d}$ are equal, that is, $\mathbf{D}_{\mathbf{u}}\mathbf{d}=0$. In this case, $\varphi(\mathbf{u})=a^{\boldsymbol{\alpha}_i\mathbf{d}}$ for each $1\leq i\leq m$.

For \ref{item 13}, it is easy to verify that, for each $1\leq i\leq m$,
\[
\psi(\bw_i) =
\begin{cases}
 (\balpha_i\bd, \overline{\balpha_i\bg}),& \balpha_i\bd\leq p,\\
 0,& \balpha_i\bd> p.
\end{cases}
\]
By flatness of addition, \(\psi(\mathbf{u})\neq 0\) if and only if all the summands of the
first component are nonzero and equal, and the summands of the second
component all coincide; equivalently, if and only if their first
coordinates are at most \(p\) and agree
(\(\mathbf{D}_{\mathbf{u}}\mathbf{d}=\mathbf{0}\)), and their second
coordinates agree
(\(\mathbf{D}_{\mathbf{u}}\mathbf{g}\equiv \mathbf{0}\pmod p\)).
Hence the three conditions are necessary and sufficient, and the common
value is
\((\boldsymbol{\alpha}_i\mathbf{d},\overline{\boldsymbol{\alpha}_i\mathbf{g}})\)
for each \(1\leq i\leq m\).
\end{proof}

\begin{lemma}\label{lem:validity}
Let $\mathbf{u}\approx\mathbf{v}$ be an identity such that
$c(\mathbf{u})\cup c(\mathbf{v})=\{x_1,\ldots,x_n\}$.
Then $\flat(\mathbb{N})$ satisfies $\mathbf{u}\approx\mathbf{v}$ if
and only if
\[
\ker_{\mathbb{N}}\mathbf{D}_{\mathbf{u}}=\ker_{\mathbb{N}}\mathbf{D}_{\mathbf{v}}
\]
and $\boldsymbol{\delta}\mathbf{d}=0$ for all
$\mathbf{d}\in \ker_{\mathbb{N}}\mathbf{D}_{\mathbf{u}}$, where
$\boldsymbol{\delta}=\boldsymbol{\alpha}-\boldsymbol{\beta}$,
$\boldsymbol{\alpha}$ is the exponent vector of any word in
$\mathbf{u}$, and $\boldsymbol{\beta}$ is the exponent vector of any
word in $\mathbf{v}$.
\end{lemma}

\begin{proof}
Suppose that
\(\flat(\mathbb{N})\vDash\mathbf{u}\approx\mathbf{v}\). Take any
\(\mathbf{d}=(d_1,\ldots,d_n)\in\ker_{\mathbb{N}}\mathbf{D}_{\mathbf{u}}\), and
define the
assignment \(\varphi\colon\{x_1,\ldots,x_n\}\to \flat(\mathbb{N})\)
by \(\varphi(x_i)=a^{d_i}\), \(1\leq i\leq n\). By
Lemma~\ref{lem:evaluation}\ref{item 12},
\(\varphi(\mathbf{u})=a^{\boldsymbol{\alpha}\mathbf{d}}\). Then
\[
\varphi(\mathbf{v})=\varphi(\mathbf{u})
=a^{\boldsymbol{\alpha}\mathbf{d}}\neq0,
\]
since \(\flat(\mathbb{N})\vDash\mathbf{u}\approx\mathbf{v}\).
Applying Lemma~\ref{lem:evaluation}\ref{item 12} again, we get
\(\mathbf{d}\in\ker_{\mathbb{N}}\mathbf{D}_{\mathbf{v}}\) and
\(\varphi(\mathbf{v})=a^{\boldsymbol{\beta}\mathbf{d}}\). Hence
\(\boldsymbol{\alpha}\mathbf{d}=\boldsymbol{\beta}\mathbf{d}\), and
so \(\boldsymbol{\delta}\mathbf{d}=0\). Interchanging \(\mathbf{u}\)
and \(\mathbf{v}\) proves the forward direction. Thus
\(\ker_{\mathbb{N}}\mathbf{D}_{\mathbf{u}}=\ker_{\mathbb{N}}\mathbf{D}_{\mathbf{v}}\) and
\(\boldsymbol{\delta}\mathbf{d}=0\) for all
\(\mathbf{d}\in \ker_{\mathbb{N}}\mathbf{D}_{\mathbf{u}}\).
Notice that the case
\(\ker_{\mathbb{N}}\mathbf{D}_{\mathbf{u}} = \ker_{\mathbb{N}}\mathbf{D}_{\mathbf{v}} =\emptyset\) is included in the above
argument.

Conversely, assume that
\(\ker_{\mathbb{N}}\mathbf{D}_{\mathbf{u}}=\ker_{\mathbb{N}}\mathbf{D}_{\mathbf{v}}\) and
\(\boldsymbol{\delta}\mathbf{d}=0\) for all
\(\mathbf{d}\in \ker_{\mathbb{N}}\mathbf{D}_{\mathbf{u}}\). Let
\[
\varphi\colon\{x_1,x_2,\ldots,x_n\}\to \flat(\mathbb{N})
\]
be an arbitrary assignment.

We first claim that either
\(\varphi(\mathbf{u})=\varphi(\mathbf{v})=0\), or both are nonzero.
Suppose otherwise. Without loss of generality, assume that
\(\varphi(\mathbf{u})\neq0\) but \(\varphi(\mathbf{v})=0\). Then for
each \(x_i\in c(\mathbf{u})\) there is \(d_i\in \mathbb{N}\) such
that \(\varphi(x_i)=a^{d_i}\). Assume that for each
\(x_j\in c(\mathbf{v})\) there is \(d_j\geq 1\) such that
\(\varphi(x_j)=a^{d_j}\). Then
\(\mathbf{d}=(d_1,\ldots,d_n)\in
\ker_{\mathbb{N}}\mathbf{D}_{\mathbf{u}}=\ker_{\mathbb{N}}\mathbf{D}_{\mathbf{v}}\),
and so \(\varphi(\mathbf{v}) \neq 0\), which is a contradiction.
This implies that there exists
\(x_{i_0}\in c(\mathbf{v})\setminus c(\mathbf{u})\) with
\(\varphi(x_{i_0})=0\).

For each \(k\geq 1\), define the assignment
\(\psi_k\colon\{x_1,\ldots,x_n\}\to \flat(\mathbb{N})\) by
\[
\psi_k(x_i)=
\begin{cases}
a^{d_i}, & x_i\in c(\mathbf{u}),\\
a^k,     & x_i\in c(\mathbf{v})\setminus c(\mathbf{u}).
\end{cases}
\]
Let \(\mathbf{d}_k=(t_{k1},\ldots,t_{kn})\) be such that
\(\psi_k(x_i)=a^{t_{ki}}\). Thus
\(\psi_k(\mathbf{u})=a^{\boldsymbol{\alpha}\mathbf{d}_k}
=\varphi(\mathbf{u})\), so
\(\mathbf{d}_k\in \ker_{\mathbb{N}}\mathbf{D}_{\mathbf{u}}\). Since
\(\ker_{\mathbb{N}}\mathbf{D}_{\mathbf{u}}=\ker_{\mathbb{N}}\mathbf{D}_{\mathbf{v}}\), we
have
\(\mathbf{d}_k\in \ker_{\mathbb{N}}\mathbf{D}_{\mathbf{v}}\). Also
\[
\boldsymbol{\delta}\mathbf{d}_k
=\boldsymbol{\alpha}\mathbf{d}_k-\boldsymbol{\beta}\mathbf{d}_k=0,
\]
which gives
\[
\psi_k(\mathbf{v})=a^{\boldsymbol{\beta}\mathbf{d}_k}
=a^{\boldsymbol{\alpha}\mathbf{d}_k}=\varphi(\mathbf{u}).
\]
But \(x_{i_0}\in c(\mathbf{v})\) and
\(\psi_k(x_{i_0})=a^k\). As \(k\) varies,
\(\psi_k(\mathbf{v})\) varies accordingly, and so it cannot remain
equal to \(\varphi(\mathbf{u})\) for all \(k\), a contradiction.
Hence the claim holds.

Now suppose that
\(\varphi(\mathbf{u}),\varphi(\mathbf{v})\neq0\). Then by
Lemma~\ref{lem:evaluation}, for each \(1\leq i\leq n\), there is
\(d_i>0\) such that \(\varphi(x_i)=a^{d_i}\). Let
\(\mathbf{d}=(d_1,\ldots,d_n)\). Then
\(\varphi(\mathbf{u})=a^{\boldsymbol{\alpha}\mathbf{d}}\) and
\(\varphi(\mathbf{v})=a^{\boldsymbol{\beta}\mathbf{d}}\). Since
\(\boldsymbol{\delta}\mathbf{d}=0\), it follows that
\(\varphi(\mathbf{u})=\varphi(\mathbf{v})\). This implies that
\(\flat(\mathbb{N})\vDash\mathbf{u}\approx\mathbf{v}\).
\end{proof}

Combining Lemmas~\ref{lem:evaluation} and~\ref{lem:validity},
we obtain the following remark.
\begin{remark}\label{remarkaap}
Suppose that $\flat(\mathbb N)\vDash \bu\approx \bv$. Then the following statements hold.
\begin{enumerate}[label=$(\arabic*)$]
\item \label{item 26092501} If
\(\ker_{\mathbb{N}}\mathbf{D}_{\mathbf{u}}=\emptyset\), then \(\mathbf{u}\) and
\(\mathbf{v}\) evaluate to \(0\) under every assignment into
\(\flat(\mathbb{N})\) or \(A_p\).

\item \label{item 26092502} If
\(\ker_{\mathbb{N}}\mathbf{D}_{\mathbf{u}}\neq\emptyset\), then
\(c(\mathbf{u}) = c(\mathbf{v})\).
\end{enumerate}
\end{remark}

We shall use the following standard fact from linear algebra.

\begin{lemma}\label{lem:rowker}
Let \(\mathbf{D}\) be an \(m\times n\) matrix over \(\mathbb{Q}\),
where \(m,n\geq 1\). Then
\[
\row_{\mathbb{Q}}(\mathbf{D})
= (\ker_{\mathbb{Q}} \mathbf{D})^{\perp},
\]
where
\[
\row_{\mathbb{Q}}(\mathbf{D})
= \{\by\mathbf{D} \mid \by \in \mathbb{Q}^m\},
\qquad
\ker_{\mathbb{Q}} \mathbf{D}
= \{\bz \in \mathbb{Q}^n \mid \mathbf{D}\bz = \bzero\},
\]
and
\[
(\ker_{\mathbb{Q}} \mathbf{D})^{\perp}
= \{\balpha\in \mathbb{Q}^n \mid
\balpha\bbeta=0 \text{ for all }
\bbeta\in \ker_{\mathbb{Q}} \mathbf{D}\}.
\]
\end{lemma}

\begin{lemma}\label{lem:kernel}
Suppose that $\flat(\mathbb N)\vDash \bu\approx \bv$ with $\ker_{\mathbb N} \D_{\bu}\neq \emptyset$.
Then
\[
\bdelta \in \row_{\mathbb{Q}}(\mathbf{D}_{\mathbf{u}})
= \row_{\mathbb{Q}}(\mathbf{D}_{\mathbf{v}}).
\]
Equivalently, there exist rational matrices \(\mathbf{R},
\mathbf{S}\) and a rational vector \(\blambda\) such that
\[
\mathbf{D}_{\mathbf{v}} = \mathbf{R} \mathbf{D}_{\mathbf{u}},\qquad
\mathbf{D}_{\mathbf{u}} = \mathbf{S} \mathbf{D}_{\mathbf{v}},\qquad
\bdelta = \blambda \mathbf{D}_{\mathbf{u}}.
\]
\end{lemma}

\begin{proof}
Since
\(\ker_{\mathbb{N}}\mathbf{D}_{\mathbf{u}}\neq \emptyset\), fix an arbitrary
vector
\(\bc \in \ker_{\mathbb{N}}\mathbf{D}_{\mathbf{u}}\). For each
\(\bz \in \ker_{\mathbb{Q}}\mathbf{D}_{\mathbf{u}}\), choose an
integer \(k \geq 1\) such that \(k\bz \in \mathbb{Z}^n\). Since
\(\bc\in \mathbb{N}^n\), we may choose \(m\) large enough that
\(m\bc + k\bz\in \mathbb{N}^n\). By linearity of
\(\mathbf{D}_{\mathbf{u}}\), both \(\bc\) and \(m\bc + k\bz\) lie in
\(\ker_{\mathbb{N}}\mathbf{D}_{\mathbf{u}} = \ker_{\mathbb{N}}\mathbf{D}_{\mathbf{v}}\).
Therefore
\[
k\mathbf{D}_{\mathbf{v}}\bz
= \mathbf{D}_{\mathbf{v}}(k\bz)
= \mathbf{D}_{\mathbf{v}}(m\bc + k\bz)
  - m\mathbf{D}_{\mathbf{v}}\bc
= 0.
\]
Thus \(\mathbf{D}_{\mathbf{v}}\bz = 0\), and hence $\bz\in \ker_{\mathbb{Q}}\mathbf{D}_{\mathbf{v}}$.
This implies that  \(\ker_{\mathbb{Q}}\mathbf{D}_{\mathbf{u}} \subseteq
\ker_{\mathbb{Q}}\mathbf{D}_{\mathbf{v}}\); the reverse inclusion
follows by symmetry. Thus
\(\ker_{\mathbb{Q}}\mathbf{D}_{\mathbf{u}} =
\ker_{\mathbb{Q}}\mathbf{D}_{\mathbf{v}}\). By
Lemma~\ref{lem:rowker}, we deduce that
\[
\row_{\mathbb{Q}}(\mathbf{D}_{\mathbf{u}})
= (\ker_{\mathbb{Q}}\mathbf{D}_{\mathbf{u}})^{\perp}
= (\ker_{\mathbb{Q}}\mathbf{D}_{\mathbf{v}})^{\perp}
= \row_{\mathbb{Q}}(\mathbf{D}_{\mathbf{v}}).
\]
Similarly, since both \(m\bc + k\bz\) and \(\bc\) lie in
\(\ker_{\mathbb{N}}\mathbf{D}_{\mathbf{u}}\), and since
\(\bdelta\bd = 0\) for all
\(\bd\in \ker_{\mathbb{N}}\mathbf{D}_{\mathbf{u}}\), we deduce that
\[
k\bdelta\bz
= \bdelta(k\bz)
= \bdelta(m\bc + k\bz) - m\bdelta(\bc)
= 0 - 0
= 0.
\]
Thus \(\bdelta \bz = 0\) for every \(\bz \in
\ker_{\mathbb{Q}}\mathbf{D}_{\mathbf{u}}\), and so
\(\bdelta \in (\ker_{\mathbb{Q}}\mathbf{D}_{\mathbf{u}})^{\perp}
= \row_{\mathbb{Q}}(\mathbf{D}_{\mathbf{u}})\).
\end{proof}

\begin{lemma}\label{lem:transfer}
Let \(\sigma\) be an identity satisfied by \(\flat(\mathbb{N})\).
Then there exists \(N_\sigma \geq 1\) such that \(A_p\) satisfies
\(\sigma\) for all primes \(p > N_\sigma\).
\end{lemma}

\begin{proof}
We may assume that \(\sigma\) is \(\mathbf{u} \approx \mathbf{v}\).
Since \(\flat(\mathbb{N})\) satisfies
\(\mathbf{u} \approx \mathbf{v}\), Lemma~\ref{lem:validity} gives
\(\ker_{\mathbb{N}}\mathbf{D}_{\mathbf{u}} = \ker_{\mathbb{N}}\mathbf{D}_{\mathbf{v}}\)
and \(\boldsymbol{\delta}\mathbf{d} = 0\) for all
\(\mathbf{d} \in \ker_{\mathbb{N}}\mathbf{D}_{\mathbf{u}}\), where
\(\boldsymbol{\delta} = \boldsymbol{\alpha} -
\boldsymbol{\beta}\), with \(\boldsymbol{\alpha}\) the exponent
vector of any word in \(\mathbf{u}\) and \(\boldsymbol{\beta}\)
the exponent vector of any word in \(\mathbf{v}\).

If \(\ker_{\mathbb{N}}\mathbf{D}_{\mathbf{u}} = \emptyset\), then by
Remark~\ref{remarkaap}\ref{item 26092501},
\(\varphi(\mathbf{u}) = \varphi(\mathbf{v}) = 0\) for every
assignment \(\varphi \colon c(\mathbf{u}) \cup c(\mathbf{v}) \to
A_p\). Hence \(\sigma\) holds in \(A_p\) for every prime \(p\), so
we may take \(N_\sigma = 1\).

Assume now that
\(\ker_{\mathbb{N}}\mathbf{D}_{\mathbf{u}} \neq \emptyset\). By
Remark~\ref{remarkaap}\ref{item 26092502},
\(c(\mathbf{u}) = c(\mathbf{v})\); denote this common content by
\(\{x_1,\ldots,x_n\}\). Also, Lemma~\ref{lem:kernel} yields
rational matrices \(\mathbf{R}, \mathbf{S}\) and a rational vector
\(\boldsymbol{\lambda}\) such that
\[
\mathbf{D}_{\mathbf{v}} = \mathbf{R}\mathbf{D}_{\mathbf{u}},\qquad
\mathbf{D}_{\mathbf{u}} = \mathbf{S}\mathbf{D}_{\mathbf{v}},\qquad
\boldsymbol{\delta} = \boldsymbol{\lambda}\mathbf{D}_{\mathbf{u}}.
\]
Let \(N_\sigma\) be a positive common multiple of all denominators
appearing in \(\mathbf{R}\), \(\mathbf{S}\), and
\(\boldsymbol{\lambda}\). If \(p > N_\sigma\), then reducing the above equalities modulo \(p\)
gives
\[
\overline{\mathbf{D}}_{\mathbf{v}}
= \overline{\mathbf{R}}\,\overline{\mathbf{D}}_{\mathbf{u}},\qquad
\overline{\mathbf{D}}_{\mathbf{u}}
= \overline{\mathbf{S}}\,\overline{\mathbf{D}}_{\mathbf{v}},\qquad
\overline{\boldsymbol{\delta}}
= \overline{\boldsymbol{\lambda}}\,\overline{\mathbf{D}}_{\mathbf{u}},
\]
where \(\overline{\mathbf{R}}\), \(\overline{\mathbf{S}}\),
\(\overline{\mathbf{D}}_{\mathbf{u}}\),
\(\overline{\mathbf{D}}_{\mathbf{v}}\) are matrices over
\(\mathbb{F}_p\), and \(\overline{\boldsymbol{\delta}}\) is a vector
over \(\mathbb{F}_p\); all entries are reduced modulo \(p\)
componentwise. Here \(\mathbb{F}_p=\mathbb{Z}_p\) is the finite field
with \(p\) elements; we use the notation \(\mathbb{F}_p\) to
emphasize its field structure. Consequently, we deduce that
\begin{equation}\label{modolo}
\ker_{\mathbb{F}_p}\overline{\mathbf{D}}_{\mathbf{u}}
= \ker_{\mathbb{F}_p}\overline{\mathbf{D}}_{\mathbf{v}},\qquad
\overline{\boldsymbol{\delta}\mathbf{g}} = \overline{\mathbf{0}}
\quad (\overline{\mathbf{g}} \in
\ker_{\mathbb{F}_p}\overline{\mathbf{D}}_{\mathbf{u}}),
\end{equation}
where
\[
\ker_{\mathbb{F}_p}\overline{\mathbf{D}}_{\mathbf{u}}
= \{\overline{\bd} \in \mathbb{F}_p^n \mid
\overline{\mathbf{D}}_{\mathbf{u}} \overline{\bd}
= \overline{\mathbf{0}}\}
\]
and
\[
\ker_{\mathbb{F}_p}\overline{\mathbf{D}}_{\mathbf{v}}
= \{\overline{\bd} \in \mathbb{F}_p^n \mid
\overline{\mathbf{D}}_{\mathbf{v}} \overline{\bd}
= \overline{\mathbf{0}}\}.
\]

Now fix a prime \(p > N_\sigma\) and an arbitrary assignment
\(\varphi \colon \{x_1,\ldots,x_n\} \to A_p\). If
\(\varphi(x_i) = 0\) for some \(i\), then
\(\varphi(\mathbf{u}) = \varphi(\mathbf{v}) = 0\) by
Lemma~\ref{lem:evaluation}\ref{item 11}. So assume
\(\varphi(x_i) = (d_i,\bar g_i)\) for each \(1\leq i\leq n\), and let
\(\mathbf{d} = (d_1,\dots,d_n)\), \(\mathbf{g} =
(g_1,\dots,g_n)\).

If
\(\mathbf{d} \notin \ker_{\mathbb{N}}\mathbf{D}_{\mathbf{u}}\), then by
Lemma~\ref{lem:evaluation}\ref{item 13},
\(\varphi(\mathbf{u}) = \varphi(\mathbf{v}) = 0\).

If
\(\mathbf{d} \in \ker_{\mathbb{N}}\mathbf{D}_{\mathbf{u}}\), then
\(\boldsymbol{\delta}\mathbf{d} = \boldsymbol{\alpha}\mathbf{d} -
\boldsymbol{\beta}\mathbf{d} = 0\), so
\(\boldsymbol{\alpha}\mathbf{d} = \boldsymbol{\beta}\mathbf{d}\),
denoted by \(s\). If \(s > p\), then
Lemma~\ref{lem:evaluation}\ref{item 13} gives
\(\varphi(\mathbf{u}) = \varphi(\mathbf{v}) = 0\). If \(s \leq p\),
then by the same lemma,
\[
\begin{array}{c}
\varphi(\mathbf{u}) \neq 0
\Leftrightarrow \mathbf{D}_{\mathbf{u}}\mathbf{g}
\equiv \mathbf{0} \pmod p
\Leftrightarrow \overline{\mathbf{g}}
\in \ker_{\mathbb{F}_p}\overline{\mathbf{D}}_{\mathbf{u}},
\\[2mm]
\hspace*{18em}\Updownarrow
\\[2mm]
\varphi(\mathbf{v}) \neq 0
\Leftrightarrow \mathbf{D}_{\mathbf{v}}\mathbf{g}
\equiv \mathbf{0} \pmod p
\Leftrightarrow \overline{\mathbf{g}}
\in \ker_{\mathbb{F}_p}\overline{\mathbf{D}}_{\mathbf{v}}.
\end{array}
\]
This implies that  either
\(\varphi(\mathbf{u}) = \varphi(\mathbf{v}) = 0\), or both
\(\varphi(\mathbf{u})\) and \(\varphi(\mathbf{v})\) are nonzero and
their nonzero values are
\[
(s,\overline{\boldsymbol{\alpha}\mathbf{g}})
\quad \text{and} \quad
(s,\overline{\boldsymbol{\beta}\mathbf{g}}),
\]
which are equal because
\(\overline{\boldsymbol{\delta}\mathbf{g}}
= \overline{\boldsymbol{\alpha}\mathbf{g}}
- \overline{\boldsymbol{\beta}\mathbf{g}}
= \overline{\mathbf{0}}\) by \eqref{modolo}.

Thus \(A_p\) satisfies \(\sigma\), as required.
\end{proof}

We are now ready to complete the proof of Theorem~\ref{mainthm}.

\vspace{0.3cm}
\textbf{Proof of Theorem~\ref{mainthm}.}
Suppose, for contradiction, that \(\mathcal{V}^{(n)}\) is finitely
based. Then \(\mathcal{V}^{(n)}\) has a finite equational basis
\(\Sigma\).  Since \(\mathcal{V}\) lies in
the interval
\([\mathsf{V}(\flat(\mathbb{N})),\mathcal{W}_{\mathbb{P}}]\), it
satisfies every \(\sigma_p\). As every \(\sigma_p\) is a
\(2\)-variable identity and \(n\geq 2\), it follows that
\(\mathcal{V}^{(n)}\) also satisfies \(\sigma_p\). Since \(\Sigma\)
is an equational basis of \(\mathcal{V}^{(n)}\), we obtain that
\(\Sigma\) derives \(\sigma_p\).

For any identity \(\sigma \in \Sigma\), \(\sigma\) is satisfied by
\(\flat(\mathbb{N})\), because \(\mathcal{V}^{(n)}\) contains
\(\mathsf{V}(\flat(\mathbb{N}))\) and \(\Sigma\) is a basis of
\(\mathcal{V}^{(n)}\). By Lemma~\ref{lem:transfer}, there is an
integer \(N_{\sigma}\) such that \(A_p\) satisfies \(\sigma\) for
each \(p>N_{\sigma}\). Now set
\[
N=\max\{N_\sigma \mid \sigma \in \Sigma\}.
\]
Then \(A_p\) satisfies all identities in \(\Sigma\) for every prime
\(p>N\). Since \(\Sigma\) derives \(\sigma_p\), it follows that
\(A_p\) satisfies \(\sigma_p\), which contradicts
Lemma~\ref{lem:separation}. Therefore, \(\mathcal{V}^{(n)}\) is
nonfinitely based.
\qed

\begin{proposition}\label{prop:scw}
Let \(W\) be a nonempty set of words. Then the flat semiring \(S_c(W)\)
is finitely based if and only if every word in \(W\) is either a
cube of a letter or has length at most \(2\).
\end{proposition}

\begin{proof}
This is a direct consequence of Propositions~\ref{prop:scwnilpotentfb} and~\ref{prop:scwnfbknil} and Theorem~\ref{mainthm}.
\end{proof}

\begin{proposition}\label{prop:mcwmw}
Let $S$ be a flat semiring of the form $M_c(W)$ or $M(W)$,
where $W$ contains a nonempty word.
Then $S$ is nonfinitely based.
\end{proposition}
\begin{proof}
We proceed by considering two cases.

\textbf{Case 1.} $S$ satisfies the identity $x^k\approx x^{k+1}$ for some $k\geq 3$.
Then, by Proposition~\ref{mcwmw}, $S$ is nonfinitely based.

\textbf{Case 2.} $S$ does not satisfy the identity $x^{k} \approx x^{k+1}$ for any $k\geq 3$.
Then, for each $k\geq 3$, there exists a nonempty word $\mathbf{w}_k$ such that
$\mathbf{w}_k^k\in S\setminus \{0\}$.
Consider the subsemiring $A_k$ of $S$ generated by $\mathbf{w}_k$.
Then $A_k$ is isomorphic to $\flat(\mathbb N)$ or $S(a^{t_k})$ for some $t_k\geq k$.

We distinguish two subcases.

\textbf{Subcase 2.1.} $A_k$ is isomorphic to $\flat(\mathbb{N})$ for some $k\geq 3$.
Then $\mathsf{V}(\flat(\mathbb{N}))$ is a subvariety of $\mathsf{V}(S)$.

\textbf{Subcase 2.2.} For every $k\geq 3$, $A_k$ is isomorphic to $S(a^{t_k})$ for some $t_k\geq k$.
Then, for each $k\geq 3$, we have
\[
\mathsf{V}(S(a^k)) \subseteq \mathsf{V}(S(a^{t_k}))\subseteq \mathsf{V}(S).
\]
It follows from Corollary~\ref{wwi} that
\[
\mathsf{V}(\flat(\mathbb N)) = \mathsf{V}(S(a^k)\mid k\geq 3)\subseteq \mathsf{V}(S).
\]

In both subcases, we have shown that $\mathsf{V}(\flat(\mathbb N))$ is a subvariety of $\mathsf{V}(S)$.

On the other hand, by Corollary~\ref{coro:4ssigmap}, $S$ satisfies every identity $\sigma_p$.
Hence $\mathsf{V}(S)$ is a subvariety of $\mathcal{W}_{\mathbb{P}}$.

We have shown that $\mathsf{V}(S)$ lies in the interval
$[\mathsf{V}(\flat(\mathbb N)), \mathcal{W}_{\mathbb{P}}]$.
By Theorem~\ref{mainthm}, $\mathsf{V}(S)$ is nonfinitely based.
\end{proof}

\begin{proposition}\label{prop:swnfb}
If the flat semiring \(S(W)\) is not \(k\)-nil for any
\(k\geq 1\), then \(S(W)\) is nonfinitely based.
\end{proposition}

\begin{proof}
Arguing as in Case~2 of the proof of
Proposition~\ref{prop:mcwmw}, we obtain
\(\flat(\mathbb{N})\in \mathsf{V}(S(W))\). By
Corollary~\ref{coro:4ssigmap},
\(S(W)\in \mathcal{W}_{\mathbb{P}}\). Hence
\(\mathsf{V}(S(W))\) lies in the interval
\([\mathsf{V}(\flat(\mathbb{N})),\mathcal{W}_{\mathbb{P}}]\). By
Theorem~\ref{mainthm}, \(S(W)\) is nonfinitely based.
\end{proof}

\begin{proposition}
Every variety in the interval $[\mathsf{V}(\flat(\mathbb N)),\mathsf{V}(\mathbb Z)]$ is nonfinitely based,
where $\mathbb Z$ denotes the max-plus semiring $(\mathbb Z,\max,+)$.
In particular, $\flat(\mathbb Z)$ and $\mathbb Z$ are nonfinitely based.
\end{proposition}

\begin{proof}
By Lemma~\ref{lem:max+r}, $\mathbb Z$ satisfies the identity $\sigma_p$ for each prime $p$.
Thus $\mathbb Z\in \mathcal{W}_{\mathbb P}$.
We now make the following claim: $\flat(\mathbb Z)\in \mathsf{V}(\mathbb Z)$.
Indeed, let $\bu\approx \bu+\bq$ be an arbitrary ai-semiring identity satisfied by $\mathbb Z$,
where $\bu = \bu_1+\cdots +\bu_m$, $\bu_i,\bq\in X^+$.
We show that $\flat(\mathbb Z)$ also satisfies $\bu\approx \bu+\bq$.
Let $\varphi\colon c(\bu)\cup c(\bq)\to \flat(\mathbb Z)$ be an arbitrary assignment.
If $\varphi(\bu)=\infty$, then $\varphi(\bu)=\infty = \varphi(\bu+\bq)$,
since $\infty$ is the additive absorbing element.
If $\varphi(\bu) = k\neq \infty$, then $\varphi(x)\in \mathbb Z$ for all $x\in c(\bu)$,
and $\varphi(\bu_i) = k$ for all $1\leq i\leq m$.
It remains to prove that $\varphi(\bq) =k$.
Suppose, for contradiction, that $\varphi(\bq)\neq k$.
Then $\varphi(\bq)$ must fall into one of the following three cases.

\textbf{Case 1.} $\varphi(\bq)=\infty$.
Then there exists $x\in c(\bq)$ such that $\varphi(x)=\infty$.
This shows that $c(\bq)\nsubseteq c(\bu)$.
Take an assignment $\psi_1\colon c(\bu)\cup c(\bq)\to \mathbb Z$ defined by
\[
\psi_1(x) =
\begin{cases}
0,&x\in c(\bu),\\
1,&x\in c(\bq)\backslash c(\bu).
\end{cases}
\]
Then $\psi_1(\bu) = 0$, but $\psi_1(\bq) \geq 1$.
Hence $\psi_1(\bu)\neq \psi_1(\bq)= \psi_1(\bu+\bq)$.

\textbf{Case 2.} $\varphi(\bq) = t>k$.
Take an assignment $\psi_2\colon c(\bu)\cup c(\bq)\to \mathbb Z$ defined by $\psi_2(x) = \varphi(x)$.
Note that the multiplicative semigroup of $\flat(\mathbb Z)$ restricted to $\mathbb Z$
is exactly ordinary addition of integers, which coincides with the multiplication in $\mathbb Z$.
Therefore, $\psi_2(\bu_i) =\varphi(\bu_i) = k$ and $\psi_2(\bq) = \varphi(\bq) = t$.
Hence
\[
\psi_2(\bu) = \max_{1\leq i\leq m}\psi_2(\bu_i) = k,
\]
and
\[
\psi_2(\bu+\bq) = \max\{\psi_2(\bu),\psi_2(\bq)\} = \max\{ k , t\} = t.
\]
This shows that $\psi_2(\bu)\neq \psi_2(\bu+\bq)$.

\textbf{Case 3.} $\varphi(\bq) = t < k$.
Take an assignment $\psi_3\colon c(\bu)\cup c(\bq)\to \mathbb Z$ defined by $\psi_3(x) = -\varphi(x)$.
As in Case 2, it is easy to verify that
\[
\psi_3(\bu) = -k,\qquad \psi_3(\bu+\bq) = \max\{-k,-t\} = -t.
\]
Thus $\psi_3(\bu)\neq \psi_3(\bu+\bq)$.

All three cases lead to a contradiction with $\mathbb Z\vDash \bu\approx \bu+\bq$.
Hence $\varphi(\bq)=\varphi(\bu)=k$.
This implies that $\flat(\mathbb Z)\vDash \bu\approx \bu+\bq$.
Therefore, $\flat(\mathbb Z)\in \mathsf{V}(\mathbb Z)$.

Since $\flat(\mathbb N)$ is a subsemiring of $\flat(\mathbb Z)$, we have
\[
\mathsf{V}(\flat(\mathbb N))\subseteq \mathsf{V}(\flat(\mathbb Z))
\subseteq \mathsf{V}(\mathbb Z)\subseteq \mathcal{W}_{\mathbb P}.
\]
By Theorem~\ref{mainthm}, every variety in the interval
$[\mathsf{V}(\flat(\mathbb N)),\mathsf{V}(\mathbb Z)]$ is nonfinitely based.
In particular, $\flat(\mathbb Z)$ and $\mathbb Z$ are nonfinitely based.
\end{proof}

\begin{remark}
The fact that the semirings $\flat(\mathbb Z)$ and $\mathbb Z$ are nonfinitely based
was also established in~\cite{mssemifield}.
\end{remark}

To end this section, we provide more information about $\mathsf{V}(\flat(\mathbb{N}))$,
which is not covered by Theorem~\ref{mainthm} when $n = 1$.

\begin{proposition}
Let \(\mathcal{V}\) denote the variety \(\mathsf{V}(\flat(\mathbb{N}))\).
Then \(\mathcal{V}^{(n)}\) is nonfinitely based for all \(n \geq 1\).
\end{proposition}

\begin{proof}
The case \(n\geq 2\) is established by Theorem~\ref{mainthm}.
For \(n=1\), first observe that if \(\flat(\mathbb{N})\) satisfies an
identity of the form
\(x^m \approx x^{n_1}+\cdots+x^{n_k}\), then necessarily
\(m=n_1=\cdots=n_k\) (by taking the assignment \(x\mapsto a\)); that
is, such an identity is trivial.

Suppose, for contradiction, that \(\mathcal{V}^{(1)}\) is finitely
based. By the compactness theorem of equational logic, there exists a
finite set \(\Sigma\) of one-variable nontrivial identities satisfied
by \(\flat(\mathbb{N})\) that forms an equational basis of
\(\mathcal{V}^{(1)}\). We may write
\[
\Sigma=\{\,f_i(x)\approx g_i(x)\mid 1\leq i\leq m\,\},
\]
where neither \(f_i(x)\) nor \(g_i(x)\) is a monomial.

Let \(N\) be the maximum exponent of \(x\) appearing among
\[
f_1(x),\,g_1(x),\,\ldots,\,f_m(x),\,g_m(x),
\]
and choose a prime \(p>N\). Let
\(C_p=\{1,g,g^2,\ldots,g^{p-1}\}\) be the cyclic group of order \(p\).
We claim that
\[
\flat(C_p)\vDash f_i(x)\approx g_i(x)
\]
for each \(1\leq i\leq m\). Indeed, for any assignment
\(\varphi\colon\{x\}\to\flat(C_p)\), consider the following cases.

\textbf{Case 1.} \(\varphi(x)\in\{0,1\}\). Then
\[
\varphi(f_i(x))=\varphi(g_i(x)),
\]
since \(0\) and \(1\) are both multiplicatively and additively
idempotent.

\textbf{Case 2.} \(\varphi(x)=g^k\) for some \(1\leq k\leq p-1\).
Since each \(f_i(x)\) and \(g_i(x)\) is a polynomial (not a single
word), each of \(f_i(x)\) and \(g_i(x)\) contains two monomials with
distinct exponents between \(1\) and \(N\).
We claim that \(\varphi(x^s+x^t)=0\) for any \(1\leq s<t\leq N\);
otherwise,
\[
\varphi(x^s+x^t)=g^{ks}+g^{kt}\neq 0.
\]
By the flatness of addition, \(g^{ks}=g^{kt}\), hence
\(g^{k(t-s)}=1\). By Lagrange's theorem, \(p\mid k(t-s)\).
Since \(1\leq k<p\) and \(1\leq t-s<p\), neither \(k\) nor \(t-s\)
is divisible by \(p\), so \(p\) does not divide \(k(t-s)\),
a contradiction.
Thus, for every \(1\leq i\leq m\),
\[
\varphi(f_i(x))=0=\varphi(g_i(x)).
\]
This implies that \(\flat(C_p)\) satisfies every identity in \(\Sigma\).
On the other hand, the identity
\begin{equation}\label{id26092501}
x+x^{p+1}\approx x+x^{p+1}+x^{p+2}
\end{equation}
is satisfied by \(\flat(\mathbb{N})\), since \(x+x^{p+1}\) always
evaluates to \(0\) in \(\flat(\mathbb{N})\) and \(0\) is an absorbing
element for addition, but it fails in \(\flat(C_p)\), since
\[
g+g^{p+1}=g\neq 0=g+g^2=g+g^{p+1}+g^{p+2}.
\]
Thus \(\Sigma\) cannot derive the identity \eqref{id26092501}, and so
\(\Sigma\) is not a basis for \(\mathcal{V}^{(1)}\), a contradiction.
Therefore, \(\mathcal{V}^{(1)}\) is nonfinitely based.
\end{proof}

\section{Conclusion}\label{sec:conclu}
We have completely solved the finite basis problem for the flat semirings
$S_c(W)$, $M_c(W)$ and $M(W)$
(Propositions~\ref{prop:scw} and~\ref{prop:mcwmw}),
yielding a complete classification of these algebras with respect to the finite basis property.
By contrast, $S(W)$ is nonfinitely based whenever $S(W)$ is not $k$-nil for any $k\geq 1$
(Proposition~\ref{prop:swnfb}).
By~\cite{gr2}, the only remaining open case for $S(W)$ is when $W$ contains
a word of length at least $4$ and $S(W)$ is $4$-nil.
A natural next step would be to extend our approach to other classes of flat semirings or to
investigate the finite basis problem for related algebras.

\subsection*{Acknowledgment}
The authors thank Professor Marcel Jackson for his valuable comments
and suggestions. In particular, he pointed out that both \(\flat(\mathbb{Z})\) and the
max-plus semiring \(\mathbb Z\) satisfy every identity \(\sigma_p\).
This observation motivated the general formulation of Theorem~\ref{mainthm}.
Miaomiao Ren, the corresponding author, is supported by the National
Natural Science Foundation of China (12371024, 12571020).

\end{document}